\documentclass[11pt]{amsart}
\usepackage{amsmath}
\usepackage{amssymb}
\usepackage{amsthm}
\usepackage[numbers]{natbib}
\usepackage{enumerate}
\usepackage{dsfont}
\usepackage{a4wide}
\usepackage{color}

\newcommand{\R}{{\mathbb R}}
\newcommand{\N}{{\mathbb N}}
\newcommand{\cN}{{\mathcal N}}
\newcommand{\Z}{{\mathbb Z}}

\newcommand{\C}{{\mathcal C}}
\newcommand{\D}{{\mathcal D}}
\newcommand{\F}{{\mathcal F}}
\newcommand{\G}{{\mathcal G}}
\renewcommand{\H}{{\mathcal H}}

\newcommand{\E}{\mathbb{E}}

\newcommand{\Var}{\operatorname{Var}}
\newcommand{\Cov}{\operatorname{Cov}}

\newcommand{\dconv}{\Rightarrow}

\newcommand{\ind}{\mathds{1}}

\newcommand\iteqref[1]{{\it(\ref{#1})}}

\def\weakto{\Rightarrow}

\renewcommand{\P}{\mathbb{P}}
\renewcommand{\H}{{\mathcal H}}
\renewcommand{\vec}{\boldsymbol}

\newtheorem{theorem}{Theorem}
\newtheorem{proposition}{Proposition}
\newtheorem{lemma}{Lemma}
\newtheorem{corollary}{Corollary}

\theoremstyle{definition}
\newtheorem{example}{Example}
\newtheorem{remark}{Remark}

\newcommand{\comment}[1]{}
\def\limn{\lim_{n\to\infty}}
\def\mfa{\mbox{ for all }}
\def\calB{{\mathcal B}}
\def\vv#1{\vec #1}
\def\indd#1{{\ind}_{\{#1\}}}

\def\inv{^{-1}}
\def\mmas{\mbox{ as }}
\def\proba{\mathbb P}
\def\pp#1{\left(#1\right)}
\def\spp#1{(#1)}

\def\ccbb#1{\left\{#1\right\}}
\def\abs#1{\left|#1\right|}
\def\Rd{{\mathbb R^d}}

\def\mmid{\;\middle\vert\;}
\def\wt#1{\widetilde{#1}}
\def\wb#1{\overline{#1}}
\def\topp#1{^{(#1)}}
\newcommand{\eqnh}{\begin{eqnarray*}}
\newcommand{\eqne}{\end{eqnarray*}}
\newcommand{\eqnhn}{\begin{eqnarray}}
\newcommand{\eqnen}{\end{eqnarray}}
\newcommand{\equh}{\begin{equation}}
\newcommand{\eque}{\end{equation}}

\begin{document}\sloppy
\title{Invariance principles for operator-scaling Gaussian random fields}

\author{Hermine Bierm\'e}
\address{
Hermine Bierm\'e\\
Laboratoire de Mathématiques et Applications, UMR-CNRS 7348\\
Universit\'e de Poitiers, T\'el\'eport 2-BP30179, Boulevard Marie et Pierre Curie,
86962 Chasseneuil, France.
}
\email{hermine.bierme@math.univ-poitiers.fr}

\author{Olivier Durieu}
\address{
Olivier Durieu\\
Laboratoire de Math\'ematiques et Physique Th\'eorique, UMR-CNRS 7350\\
F\'ed\'eration Denis Poisson, FR-CNRS 2964\\
Universit\'e Fran\c{c}ois--Rabelais de Tours, Parc de Grandmont, 37200 Tours, France.
}
\email{olivier.durieu@lmpt.univ-tours.fr}

\author{Yizao Wang}
\address
{
Yizao Wang\\
Department of Mathematical Sciences\\
University of Cincinnati\\
2815 Commons Way\\
Cincinnati, OH, 45221-0025, USA.
}
\email{yizao.wang@uc.edu}

\date{\today}

\begin{abstract}
Recently, \citet{hammond13power} introduced a model of correlated random walks that scale to fractional Brownian motions with long-range dependence. In this paper, we consider a natural generalization of this model to dimension $d\geq 2$. We define a $\Z^d$-indexed random field with dependence relations governed by an underlying random graph with vertices $\Z^d$, and we study the scaling limits of the partial sums of the random field over rectangular sets.
An interesting phenomenon appears: depending on how fast the rectangular sets increase along different directions, different random fields arise in the limit. In particular, there is  a critical regime where the limit random field is operator-scaling  and inherits the full dependence structure of the discrete model, whereas in other regimes the limit random fields have at least one direction that has either invariant or independent increments, no longer reflecting the dependence structure in the discrete model. The limit random fields form a general class of operator-scaling Gaussian random fields. Their increments and path properties are investigated.
\end{abstract}

\maketitle


\section{Introduction}

Self-similar processes are important in probability theory because of their connections with limit theorems and their intensive use in modeling, see for example \citep{taqqu86bibliographical}. These are processes $(X_t)_{t\in\R}$ that satisfy, for some $H>0$, 
\begin{equation}\label{eq:ss}
(X(\lambda t))_{t\in\R} \stackrel{fdd}=\lambda ^H(X(t))_{t\in\R}, \mbox{ for all }\lambda >0,
\end{equation}
where `$\stackrel{fdd}=$' stands for `equal in finite-dimensional distributions'. 
It is well known that the only Gaussian processes that are self-similar and have stationary increments are the fractional Brownian motions. Throughout, we let $(B_H(t))_{t\in\R}$ denote a fractional Brownian motion with Hurst index $H\in(0,1)$; this is a zero-mean Gaussian process with covariances given by
\[
 \Cov(B_H(t),B_H(s))=\frac1{2}( |t|^{2H}+|s|^{2H}-|t-s|^{2H}),\quad t,s\in\R.
\]
Fractional Brownian motions were first introduced in 1940 by \citet{kolmogorov40wienersche} and their relevance was first recognized by \citet{mandelbrot68fractional}, who gave them their name.
Invariance principles for fractional Brownian motions have a long history, since the seminal work of \citet{davydov70invariance} and \citet{taqqu75weak}. As the limiting objects of stochastic models, fractional Brownian motions have appeared in various areas, including random walks in random environment \citep{enriquez04simple}, telecommunication processes \citep{mikosch07scaling}, interacting particle systems \citep{peligrad08fractional}, and finance \citep{kluppelberg04fractional}, just to mention a few. 

Recently, \citet{hammond13power} proposed a simple discret model that scales to fractional Brownian motions with $H>1/2$. This model, to be described below, can be interpreted as  a strongly correlated random walk with $\pm1$ jumps. As the simple random walk can be viewed as the discrete counterpart of the Brownian motion, the correlated random walks proposed in \citep{hammond13power} can be viewed as the discrete counterparts of the fractional Brownian motions for $H>1/2$. In this regime, the fractional Brownian motion is well known to exhibit long-range dependence \citep{samorodnitsky06long}.

The aim of the present paper is to generalize the Hammond--Sheffield model to dimension $d\geq 2$ and to study the scaling limits. 
Based on a natural generalization of the Hammond--Sheffield model, we establish invariance principles for a new class  of operator-scaling Gaussian random fields.
 The operator-scaling random fields are generalization of self-similar processes~\eqref{eq:ss} to random fields, proposed by \citet{bierme07operator}. Namely, for a matrix $E$ with all eigenvalues having real positive parts, the random field $(X(\vec t))_{\vec t\in\R^d}$ is said to be $(E,H)$-operator-scaling for some $H>0$, if
 \begin{equation}\label{eq:OS}
 (X(\lambda^E\vec t))_{\vec t\in\R^d} \stackrel{fdd}=\lambda^H(X(\vec t))_{\vec t\in\R^d} \mbox{ for all }\lambda>0,
 \end{equation}
 where $\lambda^E:=\sum_{k\geq 0}(\log \lambda)^kE^k/k!$. In this paper, we focus on the case that $E$ is a $d\times d$ diagonal matrix with diagonal entries $\beta_1,\dots,\beta_d$, denoted by $E = {\rm diag}(\beta_1,\dots,\beta_d)$. It is worth mentioning that a simple generalization of the self-similarity would be to take $E$ being the identity matrix in~\eqref{eq:OS}, and the advantage of taking a general diagonal matrix is to be able to  accommodate  anisotropic random fields. Examples of operator-scaling Gaussian random fields include fractional Brownian sheets \citep{kamont96fractional} and L\'evy Brownian sheets \citep{samorodnitsky94stable}. 
Here, our results provide a new class to this family with corresponding invariance principles. 
 We also mention that there are other well investigated generalizations of fractional Brownian motions to Gaussian random fields, including distribution-valued ones. See for example \citep{bierme10selfsimilar,lodhia14fractional,sheffield07gaussian}.

\medskip
We now give a brief description of the Hammond--Sheffield model and its generalization to high dimensions. Let us start with the one-dimensional model. 
Let $\mu$ be a probability distribution with support in $\{1,2,\ldots\}$. Using the sites of $\Z$ as vertices, one defines a random directed graph $\G_\mu$ by sampling independently one directed edge on each site. The edge starting at the site $i\in\Z$ will point backward to the site $i-Z_i$, where $Z_i$ is a random variable with distribution $\mu$.
Here, $\mu$ is a probability distribution in form of
\begin{equation}\label{eq:RV}
\mu(\{n,\ldots\})=n^{-\alpha}L(n),
\end{equation}
 where $L$ is a slowly varying function and $\alpha\in(0,1/2)$. This choice of $\alpha$ guarantees that the graph $\G_\mu$ has a.s.\ infinitely many components, each being a tree with infinite vertices. Conditioning on $\G_\mu$, 
one then defines $(X_j)_{j\in\Z}$ such that 
\begin{itemize}
\item $X_j=X_i$ if $j$ and $i$ are in the same component of the graph, 
\item $X_j$ and $X_i$ are independent otherwise, and 
\item marginally each $X_i$ has the distribution $(1-p)\delta_ {-1}+p\delta_1$ for some $p\in(0,1)$.
\end{itemize}
The partial-sum process $S_n=\sum_{i=1}^n X_i,{n\ge1}$, can be interpreted as a correlated random walk. \citet[Theorem 1.1]{hammond13power} proved that 
\equh\label{eq:HS}
\pp{\frac{S_{\lfloor nt \rfloor}-\E S_{\lfloor nt\rfloor}}{n^{\alpha+1/2}L(n)}}_{t\in[0,1]} \weakto \sigma\pp{B_{\alpha+1/2}(t)}_{t\in[0,1]}
\eque
as $n\to\infty$.  \citet{hammond13power} actually established a strong invariance principle for this convergence. 

\medskip

To  generalize the Hammond--Sheffield model to high dimensions, we start by constructing a random graph $\G_\mu$ with vertices $\Z^d$. Similarly, at each vertex $\vec i\in\Z^d$ we first sample independently a random edge of length $\vec Z_{\vec i}$, according to a probability distribution $\mu$, and connect $\vec i$ to $\vec i-\vec Z_{\vec i}$. The distribution $\mu$ has support within $\{1,2,\dots\}^d$, intuitively meaning that all the edges are directing towards the southwest when $d=2$. Most importantly, the distribution $\mu$ is assumed to be in the strict domain of normal attraction of $(E,\nu)$, denoted by $\mu\in\D(E,\nu)$, for some matrix $E = {\rm diag}(1/\alpha_1,\dots,1/\alpha_d)$ with $\alpha_i\in(0,1), i=1,\dots,d$, and an infinitely divisible probability measure $\nu$ on $\R_+^d$.
That is, if $(\vec \xi_i)_{i\ge1}$ are i.i.d.\ copies with distribution $\mu$, then 
\equh\label{DoA}
n^{-E}\sum_{i=1}^n \vec \xi_i \weakto \nu.
\eque
 This assumption is a natural generalization of~\eqref{eq:RV} to high dimensions. 
We again focus on the case that $\G_\mu$ has infinitely many components, which turns out to be exactly the case that $q(E) := {\rm trace}(E)>2$, and given $\G_\mu$ we define $(X_j)_{j\in\Z^d}$ similarly as in dimension one. Remark that $q(E)>2$ is trivially satisfied for $d\geq 2$, due to the restriction on $\alpha_i\in(0,1)$. See Section~\ref{sec:model} for detailed descriptions of the measure $\mu$, the random graph $\G_\mu$, and the model.

\medskip

To study the scaling limit of the partial sums of the random fields over increasing rectangles, we introduce 
\[
S_{\vec n}(\vec t):=\sum_{\vec j\in R(\vec n,\vec t)}X_{\vec j}, \quad\quad{\vec n}=(n_1,\ldots,n_d)\in \N^d, {\vec t}=(t_1,\ldots,t_d)\in [0,1]^d
\]
with $R(\vec n, \vec t)=\prod_{k=1}^d[0,n_kt_k-1]\cap \Z^d$.
Surprisingly, the picture is much more complicated in high dimensions. 
In order to obtain an invariance principle for $S_{\vec n}(\vec t)$, one cannot simply require $\min_{i=1,\ldots,d} n_i \to \infty$ as most of the limit theorems for random fields do (see e.g.~\citep{bierme14invariance,dedecker01exponential,lavancier07invariance}). Instead, one needs to investigate 
\begin{equation}\label{eq:E'}
S_{n}^{E'}(\vec t):=\sum_{\vec j\in R(n^{E'}\vec 1,\vec t)}X_{\vec j}
\end{equation}
 with a diagonal matrix
$E'=\mbox{diag}(\beta_1,\ldots,\beta_d)$. 

Our main result, Theorem~\ref{thm:main}, reveals the following surprising phenomenon. For different $E'$, the limiting random field may not be the same. However, in the special case with $E' =c E$ for some $c>0$, the dependence structure of the limiting random field is determined by the measure $\nu$. This case is referred to as the critical regime. For the non-critical regime, one can still obtain invariance principles under different normalizations depending on both $E$ and $E'$, although the limiting random field has degenerate dependence structure (either invariant, i.e.~completely dependent, or independent increments) along at least one direction. Below we briefly summarize the phenomenon of critical regime. 
To the best of our knowledge, the existence of such a critical regime has been rarely seen in the literature, except for the recent results by 
\citet{puplinskaite13aggregation,puplinskaite15scaling}. They investigated invariance principles for a different random field model in dimension 2, and referred to the same phenomenon as the scaling-transition phenomenon.

\medskip

\noindent{\it Critical regime:} Here we refer to the case of taking $E' = E$ in~\eqref{eq:E'}. 
\begin{theorem}\label{thm:critical}
Assume that $\mu\in\D(E,\nu)$ for some $E = {\rm diag}(1/\alpha_1,\dots,1/\alpha_d)$ with $\alpha_i\in(0,1), i=1,\dots,d$, and a probability measure $\nu$ on $\R_+^d$.  Assume $\alpha_1<1/2$ if $d=1$. Let $\psi$ be the characteristic function of $\nu$. Then,
\[
\left(\frac{S_n^{E}(\vec t) - \E S_n^E(\vec t)}{n^{1+q(E)/2}}\right)_{\vec t\in[0,1]^d}\weakto (W(\vec t))_{\vec t\in[0,1]^d},
\]
in the space $D([0,1]^d)$, 
where the limit Gaussian random field $(W(\vec t))_{\vec t \in\R^d}$ has zero-mean and  covariance function 
\[
\Cov(W(\vec t),W(\vec s))
= {\sigma_X^2}\int_{\R^d}\prod_{k=1}^d\frac{\left(e^{it_k y_k}-1\right)\overline{\left(e^{is_k y_k}-1\right)}}{2\pi|y_k|^2}|\log\psi(\vec y)|^{-2} d\vec y,\vec t,\vec s\in\R^d,
\]
where an explicit expression of $\sigma_X^2$ is given in~\eqref{eq:sigma} below.
\end{theorem}

The limit Gaussian random field is easily seen to be $(E,H)$-operator-scaling with $H = 1+q(E)/2$.
For this new class of random fields, we study its increments and the H\"older regularity of the sample paths in Section \ref{sec:property}. 

\medskip

\noindent{\it Non-critical regime:}
For the case $E'$ in~\eqref{eq:E'} is not a multiple of $E$, the situation becomes much more subtle. One can still obtain invariance principles with appropriate normalization depending on both $E$ and $E'$. However, in the non-critical regime the limiting random fields 
no longer reflects fully the long-range dependence inherited from $\G_\mu$. In particular,
 the covariance function of the limiting random field becomes degenerate in certain directions: along these directions, the covariance function becomes the one of a fractional Brownian motion with either $H = 1/2$ (the standard Brownian motion, which is memoryless) or $H = 1$ (the case of complete dependence with $W(t) = tZ, t\ge 0$ for a common standard Gaussian random variable $Z$). Accordingly, along these directions the increments of the Gaussian random fields are independent or translation invariant, respectively. A general invariance principle is established in Section~\ref{sec:WIP}, and properties of the limiting random fields are investigated in Section~\ref{sec:property}. Here we only state the invariance principle for $d=2$. In the non-critical regime, the limit Gaussian random field is a fractional Brownian sheet with Hurst indices $H_1$ and $H_2$. 
 However, we do not see fractional Brownian sheets in the limit in high dimensions most of the time. A complete characterization of when fractional Brownian sheets arise is given in Proposition~\ref{prop:fBs} below.

\begin{theorem}\label{thm:non-critical}
Assume $d=2$. Let $\mu\in\D(E,\nu)$ with $E = {\rm diag}(1/\alpha_1,1/\alpha_2)$ and set $E' = {\rm diag}(1/\alpha_1,1/\alpha_2')$ with $\alpha_1,\alpha_2\in(0,1), \alpha_2\neq\alpha_2'$. Then, depending on the relation between $\alpha_1, \alpha_2$ and $\alpha_2'$, the following weak convergence may hold:
\[
\left(\frac{S_n^{E'}(\vec t) - \E S_n^{E'}(\vec t)}{n^{\beta}}\right)_{\vec t\in[0,1]^2}\weakto (W(\vec t))_{\vec t\in[0,1]^2},
\]
in the space $D([0,1]^2)$, 
where the limit Gaussian random field $(W(\vec t))_{\vec t \in\R^2}$ has zero-mean and  covariance function in form of
\equh\label{eq:product2}
\Cov(W(\vec t),W(\vec s)) = \sigma_X^2\sigma^2\Cov(B_{H_1}(t_1),B_{H_1}(s_1))\Cov(B_{H_2}(t_2),B_{H_2}(s_2)).
\eque
Here, $\beta,\sigma^2,H_1,H_2$ and hence $\{W(\vec t)\}_{\vec t\in[0,1]^d}$ all depend on $\alpha_1,\alpha_2$ and $\alpha_2'$.
In particular, there are four different possibilities as follows:

\begin{enumerate}[(i)]
\item \label{d=2:1}$\alpha_2'>\alpha_2, \alpha_2\in(0,1/2)$: $\beta = \frac{\alpha_2}{\alpha_2'}+\frac12\spp{\frac1{\alpha_1}+\frac1{\alpha_2'}}$, $H_1 = \frac12$, $H_2 = \frac12+\alpha_2$.
\item \label{d=2:2}$\alpha_2'>\alpha_2, \alpha_2\in(1/2,1)$: $\beta = 1+\frac1{2\alpha_1}+\frac1{\alpha_2'}-\frac1{2\alpha_2}, H_1=\frac12+\alpha_1\spp{1-\frac1{2\alpha_2}}$, $H_2 = 1$.
\item \label{d=2:3}$\alpha_2'<\alpha_2, \alpha_1\in(0,1/2)$: $\beta = 1+\frac12\spp{\frac1{\alpha_1}+\frac1{\alpha_2'}}$, $H_1 = \frac12+\alpha_1$, $H_2 = \frac12$.
\item \label{d=2:4}$\alpha_2'<\alpha_2, \alpha_1\in(1/2,1)$:  $\beta = \frac{\alpha_2}{\alpha_2'}\spp{1-\frac1{2\alpha_1}}+\frac1{\alpha_1}+\frac1{2\alpha_2'}$, $H_1 = 1$, $H_2 = \frac12+\alpha_2\spp{1-\frac1{2\alpha_1}}$.
\end{enumerate}
\end{theorem} 
The explicit expressions of $\sigma^2$ in these cases can be found in the proof of Theoroem~\ref{thm:non-critical} in Section~\ref{sec:property}.

The main result of the paper, Theorem~\ref{thm:main}, is a unified version of invariance principles for general $d\in\N$, $E = {\rm diag}(1/\alpha_1,\dots,1/\alpha_d)$ and arbitrary $E'$, from which both Theorems~\ref{thm:critical} and~\ref{thm:non-critical} follow as immediate corollaries. Theorem~\ref{thm:main} also provides a general principle to determine the correct normalization order, the limit covariance function, and hence the directions of degenerate  dependence. We have just seen that in dimension 2 there are already 4 different non-critical regimes. For general $d\geq 3$, the situation becomes more complicated.

The core of the proofs is an application of the martingale central limit theorem, thanks to the key observation that the random field of interest can be represented as a linear random field in form of
\begin{equation}\label{eq:linear}
X_{\vec i} = \sum_{\vec j\in\Z^d}q_{\vec j}X_{\vec i-\vec j}^*,\vec i\in\Z^d,
\end{equation}
of which the innovations $(X_{\vec j}^*)_{\vec j\in\Z^d}$ are multiparameter martingale differences.
\citet{hammond13power} also made essential use of the martingale central limit theorem, although the representation as a linear process as in~\eqref{eq:linear} was not explicit. 
This representation plays a key role in our proofs, as from there
 when verifying conditions in the martingale central limit theorem, thanks to the structure of the linear process, we can deal with the coefficients $q_{\vec j}$ and innovations $X_{\vec j}^*$ separately. This framework, or more generally the martingale approximation method, has been carried out  successfully in dimension one to establish invariance principles for fractional Brownian motions for general stationary processes \citep{dedecker11invariance}.  To extend this framework to high dimensions, a notorious difficulty is to find a convenient multiparameter martingale to work with. It is well known that the martingale approximation method applied to stationary random fields is not as powerful as to stationary sequences, as pointed out long time ago by \citet{bolthausen82central}. Fortunately, our specific model can be represented exactly as a simple linear random field with martingale-difference innovations as in~\eqref{eq:linear}. 
 
 Once the representation of linear random fields in~\eqref{eq:linear} is established, the main work lies in the computation of the limit of the covariance functions. This step is heavily based on the analysis of Fourier transforms of the linear coefficients $(q_{\vec i})_{\vec i\in\Z^d}$, the asymptotic property of which is essentially determined by $\nu$. Analyzing the Fourier transforms is a standard tool to compute the covariance functions for stationary linear random fields, see for example \citep{lavancier07invariance,puplinskaite13aggregation,puplinskaite15scaling}. To complete the invariance principle, the tightness is established. At last, to develop the sample-path properties we apply recent results in \citet{bierme09holder}.

The rest of the paper is organized as follows. In Section~\ref{sec:model} we describe in details the random-field model. Section~\ref{sec:CLT} provides a general central limit theorem that serves our purpose. Section~\ref{sec:WIP} establishes a general invariance principle that applies to both critical and non-critical regimes. 
Some  properties of the limit random fields are provided in Section~\ref{sec:property}.

\subsection*{Acknowledgement}
The third author would like to thank the hospitality of 
Laboratoire de Math\'ematiques et Physique Th\'eorique, UMR-CNRS 7350, Tours, France, during his visit from April to July in 2014, when the main part of this project was accomplished. The third author's research was partially supported by NSA grant H98230-14-1-0318.

\section{The model}\label{sec:model}
In this section, we will give a detailed description of our random field model $\{X_{\vec i}\}_{\vec i\in\Z^d}$, of which the dependence structure is determined by an underlying random graph ${\mathcal G}_\mu$. The asymptotic properties of the random graph is determined by a probability measure $\mu$ on $\{1,2,\dots\}^d$, which is assumed to be in the strict domain of normal attraction of an $E$-operator stable measure $\nu$ on $\mathbb R_+^d$. Some simple properties of the model will be derived. In particular, we show that the random field of interest can be represented as a linear random field, of which the innovations are stationary multiparameter martingale differences. 

 Throughout the paper we use the following usual notations.
Let $d\ge 1$ be an integer. On $\R^d$, we consider the partial order (also denoted by $<$) defined by ${\vec t}<{\vec s}$ if 
$t_j<s_j$ for all $j=1,\ldots,d$, where ${\vec t}=(t_1,\ldots,t_d)$ and ${\vec s}=(s_1,\ldots,s_d)$. In the same way, we use the notations $>$, $\le$, $\ge$.
We write ${\vec t}\nless{\vec s}$ as soon as $t_j\ge s_j$ for at least one $j=1,\ldots,d$, and in the same way, we use $\ngtr$, $\nleq$, $\ngeq$.
We denote by $[{\vec t},{\vec s}]$ the set $[t_1,s_1]\times\cdots\times[t_d,s_d]$ and
we write $|{\vec t}|_\infty$ for $\max\{|t_j|,\; j=1,\ldots,d\}$, and $|{\vec t}|_1$ for $\sum_{j=1}^d|t_j|$. Furthermore, write $\N = \{0,1,\dots\}$ and $\N_*=\{1,2,\dots\}$.

\subsection{The random graph}\label{randomgraph}\

Let $\mu$ be a probability measure on $\N_*^d$ such that the additive group generated by the support of $\mu$ is all $\Z^d$ (we say that $\mu$ is aperiodic).
On $\Z^d$, we consider the random directed graph $\G_\mu$, associated to $\mu$, defined as follows:
\begin{itemize}
 \item Let $({\vec Z}_{\vec n})_{\vec n\in\Z^d}$ be i.i.d.\ random variables with distribution $\mu$.
 \item For each $\vec n\in\Z^d$, let $e_{\vec n}$ be the outward edge from $\vec n$ to ${\vec n}-{\vec Z}_{\vec n}$.
 \item $\G_\mu$ is the graph with all sites of $\Z^d$ as vertices and random directed edges $\{e_{\vec n},\, \vec n\in\Z^d\}$.
\end{itemize}

The graph $\G_\mu$ is then composed of (possibly) several disconnected components and each component is a tree. The upcoming Proposition~\ref{prop1} shows that, almost surely, the number of components of $\G_\mu$ is one or is infinite.
 
We first introduce the following notations. For ${\vec n}\in\Z^d$, we denote by $A_{\vec n}$ the ancestral line of ${\vec n}$, that is the set of all elements ${\vec k}\in\Z^d$ for which there exists a directed connection from ${\vec n}$ to ${\vec k}$ (taking the orientations of the edges into account). Note that, in distribution, $A_{\vec n}$ can be described by the range of the random walk $({\vec n}-{\vec S}_{k})_{k\ge 0}$ where $({\vec S}_k)_{k\ge 0}$ is the random walk starting at ${\vec 0}$ with step distribution $\mu$. In particular, since $\mu$ is supported by $\N_*^d$, any element $\vec k$ in $A_{\vec n}$ satisfies ${\vec k}<{\vec n}$.
Observe that the condition that the support of $\mu$ generates the group $\Z^d$ is equivalent to the fact that $\P(A_{\vec n}\cap A_{\vec m}\ne \emptyset)>0$ for all ${\vec n}$, ${\vec m}\in\Z^d$.

For ${\vec n}\in\Z^d$, we set $q_{\vec n}=\P({\vec 0}\in A_{\vec n})$. We clearly have $q_{\vec n}=0$ as soon as $\vec 0\nless\vec n$, except for $q_{\vec0}=1$. Further, since each edge is generated independently at each site, for any ${\vec n}$, ${\vec k}\in\Z^d$, $$\P({\vec k}\in A_{\vec n})=q_{{\vec n}-{\vec k}}.$$

\begin{proposition}\label{prop1}
If $\sum_{\vec k\in\N^d}q_{\vec k}^2$ converges, then $\G_\mu$ has almost surely infinitely many components whereas if  $\sum_{\vec k\in\N^d}q_{\vec k}^2$ diverges, then $\G_\mu$ has almost surely only one component.
\end{proposition}
We start by proving the following lemma.
\begin{lemma}\label{lem:probainter}
\begin{enumerate}[(i)]
 \item\label{lem:probainter:1}  If $\sum_{\vec k\in\N^d}q_{\vec k}^2$ converges then for all ${\vec n}\in\Z^d$, 
$$\P(A_{\vec0}\cap A_{\vec n}\ne \emptyset)=\left(\sum_{{\vec k}\in\N^d} q_{\vec k}^2\right)^{-1}\left(\sum_{{\vec k}\in\Z^d} q_{\vec k}q_{{\vec k}+{\vec n}}\right).$$
 \item\label{lem:probainter:2}  If $\sum_{{\vec k}\in\N^d}q_{\vec k}^2$ diverges then $\P(A_{\vec 0}\cap A_{\vec n}\ne \emptyset)=1$ for all ${\vec n}\in\Z^d$.
\end{enumerate}
\end{lemma}

\begin{proof}
The proof follows an idea developed in \citep[Lemma~3.1]{hammond13power} for the dimension $1$.
Let $\G_\mu'$ be an independent copy of $\G_\mu$. We denote by $A_{\vec n}'$ the ancestral line of ${\vec n}$ with respect to   $\G_\mu'$.
On one hand, one has
$$
\E|A_{\vec 0}\cap A_{\vec n}'|=\sum_{{\vec k}\in\Z^d}\P({\vec k}\in A_{\vec 0})\P({\vec k}\in A_{\vec n}) =\sum_{{\vec k}\in\Z^d} q_{\vec k}q_{{\vec k}+{\vec n}}.
$$
On the other hand, 
$$
\E|A_{\vec 0}\cap A_{\vec n}'|=\P(A_{\vec 0}\cap A_{\vec n}\ne \emptyset)\E|A_{\vec 0}\cap A_{\vec 0}'|=\P(A_{\vec 0}\cap A_{\vec n}\ne \emptyset)\sum_{{\vec k}\in\N^d} q_{\vec k}^2
$$
and thus \iteqref{lem:probainter:1} follows.

If $\sum_{{\vec k}\in\N^d} q_{\vec k}^2 = \infty$, then $\E|A_{\vec 0}\cap A_{\vec 0}'|=\infty$. 
But $\E|A_{\vec 0}\cap A_{\vec 0}'|$ can also be computed as
\begin{equation}\label{eq:EA0A0'}
\E|A_{\vec 0}\cap A_{\vec 0}'|
= \sum_{k\ge 0}\P(|A_{\vec 0}\cap A_{\vec 0}'|>k)
=\sum_{k\ge 0}\P(A_{\vec0}\cap A_{\vec 0}'\ne \{{\vec 0}\})^k
=\frac{1}{1-\P(A_{\vec 0}\cap A_{\vec 0}'\ne \{{\vec 0}\})}.
\end{equation}
Thus $\E|A_{\vec 0}\cap A_{\vec 0}'|=\infty$ if and only if $\P(A_{\vec 0}\cap A_{\vec 0}'\ne \{{\vec 0}\})=1$, and in this situation $|A_{\vec 0}\cap A_{\vec 0}'|=\infty$ almost surely.
Now, since the group generated by the support of $\mu$ covers $\Z^d$, we know that, for all ${\vec n}\in\Z^d$, there exists
${\vec k}_0\in\Z^d$ such that
$$
\P({\vec k}_0\in A_{\vec 0}\mbox{ and }{\vec k}_0-{\vec n}\in A_{\vec 0}')=\P({\vec k}_0\in A_{\vec 0}\cap A_{\vec n}')>0.
$$
But, since $|A_{\vec 0} \cap A'_{\vec 0}|=\infty$ a.s., we infer that $|A_{\vec k_0} \cap A_{\vec k_0-\vec n}|=\infty$ also a.s., and thus
$$
\P(A_{\vec 0} \cap A_{\vec n} \neq \emptyset)= \P(A_{\vec k_0} \cap A_{\vec k_0-\vec n} \neq \emptyset)\ge\ \P(|A_{\vec k_0} \cap A_{\vec k_0-\vec n}|=\infty)=1,
$$
which proves \iteqref{lem:probainter:2}.
\end{proof}

\begin{proof}[Proof of Proposition~\ref{prop1}]
If $C:=\sum_{{\vec k}\in\N^d} q_{\vec k}^2 < \infty$, from Lemma~\ref{lem:probainter}~\iteqref{lem:probainter:1}, we get
$$
\P(A_{\vec 0}\cap A_{\vec n}\ne \emptyset)=C^{-1}\sum_{{\vec k}\in\Z^d} q_{\vec k}q_{{\vec k}+{\vec n}}\le C^{-1}\left(\sum_{{\vec k}\in\N^d,\, \vec k\ge -\vec n} q_{\vec k}^2\right)^\frac12\left(\sum_{{\vec k}\in\N^d,\, \vec k\ge \vec n} q_{\vec k}^2\right)^\frac12,
$$
which goes to $0$ as $|{\vec n}|_\infty\to \infty$.
Thus, $\P(A_{\vec 0}\cap A_{\vec n}\ne \emptyset)\to 0$ as $|{\vec n}|_\infty\to \infty$, and we can build a sequence $({\vec n}_k)_{k\in\N}\subset \Z^d$, iteratively, such that for each $k\in\N$, 
$$
\P\left(A_{{\vec n}_k} \cap (\cup_{j=0}^{k-1} A_{{\vec n}_j})\ne \emptyset\right) \le \frac{1}{k^2}.
$$
By the Borel-Cantelli lemma, we see that, almost surely, the ancestral lines $A_{{\vec n}_k}$, $k\ge 1$, are disjoint from each other. This proves the first part of the proposition.\\
 The second part of the proposition is clear from Lemma~\ref{lem:probainter}~\iteqref{lem:probainter:2}.
\end{proof}

\subsection{The measure}\label{sec:MRV}\

From now on, we always consider an aperiodic probability measure $\mu$ on $\N_*^d$ such that $\mu\in\D(E,\nu)$, as defined in~\eqref{DoA}, where $\nu$ is a full probability measure on $\R_+^d$ and $E = {\rm diag}(1/\alpha_1,\dots,1/\alpha_d)$. 
Since the distribution of each coordinate is in the strict domain of normal attraction of a positive stable laws and since positive $\alpha$-stable laws only exist for $\alpha\in(0,1)$, it necessarily follows that $\alpha_i\in(0,1)$ for all $i=1,\dots,d$.

In this case, $\mu$ is also said to have non-standard multivariate regular variation with exponent $E$. 
Equivalently, $\mu$ is (non-standard) multivariate regularly varying with exponent measure $\phi$, with $\phi$ being the same L\'evy measure in the triplet representation of $\nu$ \cite[Corollary 8.2.11]{meerschaert01limit}. That is, for some constant $c>0$,
\equh\label{eq:MRV}
\limn n\mu(n^EA)=c\phi(A) \mfa A\in\calB(\R^d) \mbox{ bounded away from $\vv 0$ and $\phi(\partial A) = 0$}.
\eque
Equivalently, this means that $n\mu(n^E\cdot)$ converges vaguely to $c\phi$, in the space of Radon measures on $\R^d\setminus\{\vv 0\}$ equipped with the vague topology, under which sets in $\calB(\R^d)$ bounded away from $\{\vv0\}$ are relatively compact. 
Most of the applications in the literature of multivariate regular variation, however, focus on the case that $\alpha_1=\cdots=\alpha_d$. In this case,~\eqref{eq:MRV} is referred to as multivariate regular variation in the literature. Standard references on (standard) multivariate regular variation includes \citep{resnick87extreme,resnick07heavy}. References on non-standard multivariate regular variation include  \citep{resnick79bivariate}, \cite[Chapter 6]{resnick07heavy}. See also some recent development in~\citep{resnick14tauberian}. Some examples are given at the end of the subsection.

We denote by $P$ the Fourier transform of the measure $\mu$, that is
$$
P({\vec t})=\sum_{{\vec k}\in\N^d}\mu(\{\vec k\}) e^{i {\vec t}\cdot{\vec k}},\qquad {\vec t}\in\R^d.
$$
Note that the assumption that the additive group generated by the support of $\mu$ is all $\Z^d$ is equivalent to:
$$
P(\vec t)=1\mbox{ if and only if the coordinates of ${\vec t}$ belong to $2\pi\Z$},
$$
see for example~\citet[p.76]{spitzer76principles}.

Let $\G_\mu$ be the random graph associated to $\mu$ as defined in Section \ref{randomgraph}.
The asymptotic behavior of $\{q_{\vec k}\}_{\vec k\in\N^d}$ will play a key role in our analysis. It is essentially determined by the measure $\mu\in\D(E,\nu)$. We denote by $Q$
the Fourier series with coefficients $q_{\vec k}=\P({\vec 0}\in A_{\vec k})$, that is 
$$
Q({\vec t})=\sum_{{\vec k}\in\N^d}q_{\vec k} e^{i{\vec t}\cdot{\vec k}}.
$$
Using that $q_{\vec k}=\sum_{\vec j\in\N_*^d}\mu(\{\vec j\})q_{\vec k-\vec j}$ for $\vec k>\vec 0$, we see that both Fourier series are linked by the relation
$$
Q(\vec t)=\frac{1}{1-P(\vec t)}.
$$
From Lemma~\ref{lem:probainter}, we see that 
$$
\P(A_{\vec0}\cap A_{\vec n}\ne \emptyset)=\frac{c_{\vec n}(|Q|^2)}{c_{\vec 0}(|Q|^2)},
$$
where $c_{\vec k}(|Q|^2)$ denotes the Fourier coefficient of index ${\vec k}$ of $|Q|^2=Q\overline{Q}$. This relation explains why the Fourier series $Q$ plays a crucial role in the study of the random graph.

We denote by $\psi(\vec t) = \int_{\R_+^d} e^{i\vec t\cdot \vec x} d\nu(\vec x)$ the characteristic function of the full $E$-operator stable measure $\nu$. Note that it follows from~\eqref{DoA} that the log-characteristic function $\log\psi$ is then an $E$-homogeneous function, that is 
$$
\mbox{for all }t>0\mbox{ and }\vec x \in\R^d,\; \log \psi(t^E\vec x)= t\log \psi(\vec x).
$$
Further, $\log\psi(\vec 0)=0$ and for all $\vec x\ne \vec 0$, $|\log\psi(\vec x)|>0$.

The two following lemmas are key results concerning the behavior of $Q$ at $\vec 0$.

\begin{lemma} \label{Qhomogene}
Let $\mu\in\D(E,\nu)$ be as described above and $\psi$ the characteristic function of $\nu$.
Then
$$|Q(\vec x)|=|1-P(\vec x)|^{-1}=\frac{g(\vec x)}{|\log\psi(\vec x)|}, \;\vec x \in [-\pi,\pi]^d,$$
where $g$ is continuous and positive with $g(\vec 0)=1$.
\end{lemma}
 
\begin{proof} Let us use a change of variables in polar coordinates.
 As in \cite[Chapter 6]{meerschaert01limit},  we define a new norm on $\R^d$, related to the matrix $E$, by 
\begin{equation}\label{def:Enorm}
 \|{\vec x}\|_E=\int_0^1 |r^E{\vec x}|\frac{1}{r}dr,
\end{equation}
where here $|\cdot|$ denotes the Euclidean norm. 
The unit ball $S_E=\{{\vec x}\in\R^d\mid \|{\vec x}\|_E=1\}$ associated to this norm is a compact set of $\R^d\setminus\{\vec 0\}$ and 
every vector in $\R^d\setminus\{{\vec 0}\}$ can be uniquely written as $r^E{\vec \theta}$ with $r>0$ and ${\vec \theta}\in S_E$, since for any $\vec x\ne \vec 0$, the map $t\mapsto \|t^E\vec x\|_E$ is strictly increasing on $(0,\infty)$.

Since $\mu\in \D(E,\nu)$, we have 
$$
P(n^{-E}\vec \theta)^n \rightarrow \psi(\vec \theta),\; \mbox { as }n\to\infty, \mbox{ uniformly in }\vec  \theta\in S_E,
$$
from which we infer that
$$
t\log P(t^{-E}\vec  \theta) \rightarrow \log\psi(\vec  \theta),\; \mbox { as }t\to\infty, \mbox{ uniformly in }\vec  \theta\in S_E,
$$
see \citep[p.159]{maejima94operator}.
Using that $\log(1+x)\sim x$ as $x\to 0$ and that $P$ is continuous at $\vec 0$, we obtain
$$
t(P(t^{-E}\vec  \theta)-1) \rightarrow \log\psi(\vec  \theta),\; \mbox { as }t\to\infty, \mbox{ uniformly in }\vec  \theta\in S_E.
$$
Thus, for all $\varepsilon >0$, there exists $T>0$ such that for all $t>T$, 
$$
\left|\frac{|\log\psi(t^{-E}\vec  \theta)|}{|P(t^{-E}\vec  \theta)-1|} -1\right|=\left|\frac{|\log\psi(\vec  \theta)|}{t|P(t^{-E}\vec  \theta)-1|} -1\right|\le \varepsilon,  \mbox{ uniformly in }\vec \theta\in S_E.
$$
Now, set $g(\cdot)=|\log\psi(\cdot)(P(\cdot)-1)^{-1}|$. The function $g$ is clearly continuous and positive on $[-\pi,\pi]^d\setminus\{\vec 0\}$. 
Set $\delta=\inf_{\vec\theta\in S_E} \| T^{-E}\vec \theta \|_E>0$. Then for all $\vec x$ such that $\|\vec x\|_E<\delta$, $\vec x=t_0^{-E}\vec \theta_0$ with $\vec \theta_0\in S_E$ and $t_0>T$ and thus  
$$
|g(\vec x)-1|=|g(t_0^{-E}\vec \theta_0)-1|\le \varepsilon.
$$
Thus $g$ is continuous at $\vec 0$ and $g(\vec 0)=1$.
\end{proof}

We are thus interested by the function $\vec x\mapsto \log\psi(\vec x)$, which is a continuous $E$-homogeneous function that only vanishes at $\vec 0$. Recall that $q(E) = {\rm trace}(E)$.
\begin{lemma}\label{integrabilite} If $\phi:\R^d\to \R$ is a continuous $E$-homogeneous function
that only vanishes at $\vec 0$,
then for any $p>0$,
$\vec x\mapsto |\phi(\vec x)|^{-p}$ is locally integrable in $\R^d$ if and only if $q(E)>p$.
\end{lemma}
\begin{proof}
There exists a unique finite Radon measure $\sigma_E$ on $S_E$ which allows the change of variable
$$
\int_{\R^d} f({\vec t}) d{\vec t}=\int_0^{+\infty}\int_{S_E}f(r^E{\vec \theta}) r^{q(E)-1} d\sigma_E({\vec \theta})dr,
$$
for all $f\in L^1(\R^d)$ (see \citep{bierme07operator}, Proposition 2.3).
Thus, using the $E$-homogeneity of $\phi$, one has
\begin{align*}
\int_{\{\|\vec x\|_E\le 1\}}|\phi(\vec x)|^{-p}d{\vec x}
&=\int_0^{1}\int_{S_E}r^{q(E)-1}|\phi(r^E{\vec \theta})|^{-p} d\sigma_E({\vec \theta}) dr \\
&=\int_0^{1}r^{q(E)-1-p} dr \int_{S_E}|\phi({\vec \theta})|^{-p} d\sigma_E({\vec \theta}).
\end{align*}
The second integral is finite because $|\phi|$ is continuous and positive on the compact set $S_E$, and the first integral is finite if and only if $q(E)>p$.
\end{proof}

As a first consequence, we get the following proposition.
\begin{proposition}
 Let $\mu\in\D(E,\nu)$. The random graph $\G_\mu$ has almost surely infinitely many components if and only if $q(E)>2$.
\end{proposition}
Note that, when $d=1$, the condition $q(E)>2$ becomes $\alpha_1<\frac12$, which corresponds to the condition assumed in \citep{hammond13power}. When $d\ge 2$, since $\alpha_i\in(0,1)$ for all $i=1,\ldots,d$, then the conditon $q(E)>2$ is always satisfied.
\begin{proof}
As a consequence of Lemma~\ref{Qhomogene}, using Parseval 
identity, we get
$$
\sum_{{\vec k}\in\N^d} q_{\vec k}^2=\frac{1}{(2\pi)^d}\int_{[-\pi,\pi]^d}|Q(\vec x)|^2 d{\vec x}= \frac{1}{(2\pi)^d}\int_{[-\pi,\pi]^d}|g(\vec x)|^2|\log\psi(\vec x)|^{-2} d{\vec x}.
$$
Since $g$ is bounded and bounded away from $0$ on any compact set, we see that $\sum_{{\vec k}\in\N^d} q_{\vec k}^2 <+\infty$ if and only if $\vec x\mapsto|\log \psi(\vec x)|^{-2}$ is integrable on $[-\pi,\pi]^d$.
The function $\vec x\mapsto\log \psi(\vec x)$ being $E$-homogeneous, by Lemma~\ref{integrabilite},
it is the case if and only if $q(E)>2$ and the result follows from Proposition~\ref{prop1}.
\end{proof}

To conclude the section, we give few examples of possible probability measure $\mu\in\D(E,\nu)$.

\begin{example}[Product measure]
Let $\mu$ be the product measure $\mu_1\otimes\cdots\otimes\mu_d$, where each $\mu_i$ is a regularly varying measure on $\N_*$ with index $\alpha_i\in(0,1)$ such that
$$
\mu_i([n,\infty))\sim c_in^{-\alpha_i},
$$
for some $c_i>0$.
Then, each $\mu_i$ belongs to the strict domain of normal  attraction (with normalization $n^{-1/\alpha_i}$) of a positive $\alpha_i$-stable law $\nu_i$, see \citep[Theorem 8.3.1]{bingham87regular}. 
Positive $\alpha$-stable laws only exist for $\alpha\in(0,1)$, and then, their characteristic functions are given by
$$
\varphi(t)=\exp\ccbb{-\gamma|t|^\alpha\pp{1-i\mbox{sgn}(t) \tan\pp{\frac{\pi}{2}\alpha}}},
$$
for some $\gamma>0$. See \citep[Theorem 8.3.2]{bingham87regular}. 
In this situation, the measure $\mu$ belongs to the strict domain of normal attraction of the measure $\nu=\nu_1\otimes\cdots\otimes\nu_d$ which is a full $E$-operator stable distribution, with $E=\mbox{diag}(1/\alpha_1,\ldots,1/\alpha_d)$. The characteristic function $\psi$ of $\nu$ is such that
$$
\log\psi(\vec x)=\sum_{j=1}^d \gamma_j|x_j|^{\alpha_j} \pp{1-i\mbox{sgn}(x_j) \tan\pp{\frac{\pi}{2}\alpha_j}},
$$
for some $\gamma_j>0$.
\end{example}

\begin{example}[Standard multivariate regular variation]\label{example:standard}
For the standard multivariate regular variation, that is when $\alpha_1=\cdots=\alpha_d=\alpha$, many examples have been known from the studies of heavy-tailed random vector $\vv X = (X_1,\dots,X_d)\in\Rd$, in the literature of heavy-tailed time series. An extensively investigated condition for multivariate regular variation is
\equh\label{eq:threshold}
\frac{\proba\pp{|{\vv X}|>ux, \vv X/|{\vv X}|\in\cdot}}{\proba(|{\vv X}|>u)} \weakto {x^{-\alpha}\sigma(\cdot)} \mmas u\to\infty, \mfa x>0,
\eque
for $|\cdot|$ a norm on $\R^d$ and $\sigma$ a probability measure on ${\mathcal B}(S)$ for $S=\{x\in\R^d: |x|=1\}$.
See for example \citep{basrak09regularly}.
It is known that~\eqref{eq:MRV} implies~\eqref{eq:threshold} (see e.g.~\citep[Theorem 1.15]{lindskog04multivariate}).

The measure $\sigma$ is often referred to as the spectral measure, which captures the dependence of extremes. For example, the case that $\sigma$ concentrates on the $d$-axis with equal mass means that, in view of~\eqref{eq:threshold}, the extremes of the stationary processes are asymptotically independent. For more theory and examples on spectral measures reflecting asymptotic dependence of the extremes, we refer to \cite[Chapter 5]{resnick87extreme}.

\end{example}

\begin{example}[Polar coordinate]\label{example:polar}
A standard procedure to obtain non-standard regularly varying random vectors is via the representation using polar coordinate. 
We use the norm $\|\cdot\|_E$ introduced in \eqref{def:Enorm} to identify $\R^d\setminus\{\vv0\}$ with $(0,\infty)\times S_E$ for the unit ball 
 $S_E=\{{\vec x}\in\R^d\mid \|{\vec x}\|_E=1\}$ such that
every vector in $\R^d\setminus\{{\vec 0}\}$ can be uniquely written as $r^E{\vec \theta}$ with $r>0$ and ${\vec \theta}\in S_E$.
By \cite[Theorem 6.1.7]{meerschaert01limit}, in case of~\eqref{DoA} (equivalently~\eqref{eq:MRV}), $\phi$ can be taken to have the polar coordinate representation
\equh\label{eq:polar}
\phi(A) = \int_0^\infty\int_{S_E}\indd{t^E\vv\theta\in A}\sigma(d\vv\theta)\frac{dt}{t^2},
\eque
for some  finite Borel measure $\sigma$ on $S_E$.  In our case, since $\mu$ has support contained in $\N^d_*$, $\phi$ is a measure on $\R^d_+$, and $\sigma$ is a finite measure on $S_E^+=S_E\cap \R^d_+$.  Identifying $\R_+^d\setminus\{\vv0\}$ with $(0,\infty)\times S_E^+$, to obtain a multivariate regular varying measure as in~\eqref{eq:MRV}, it suffices to show
\equh\label{eq:MRV2}
\mu((r,\infty)\times\Gamma)\sim cr\inv\sigma(\Gamma) \mmas r\to\infty, \mfa \Gamma\in\calB(S_E^+).
\eque
This follows from a standard argument showing that $\{(r,\infty)\times\Gamma\}_{r>0, \Gamma\in\calB(S_E^+)}$ are a convergence determining class. 

A standard procedure to construct a random vector of which the distribution $\mu$ satisfies~\eqref{eq:MRV2} is the following. Let $R$ be a non-negative random variable with $\proba(R>r)\sim c\sigma(S_E^+)r\inv$ as $r\to\infty$. Let $\vv\Theta$ be a random element in $S_E^+$ with probability $\sigma/\sigma(S_E^+)$. Assume that $R$ and $\vv\Theta$ are independent. Then, $R^E\vv\Theta$ is regularly varying in $\R^d_+$ in the sense of~\eqref{eq:MRV2}. Indeed,
\[
\proba(R^E\vv\Theta\in(r,\infty)\times\Gamma) 
= \proba(R>r, \vv\Theta\in\Gamma) \sim cr\inv\sigma(\Gamma) \mmas r\to\infty.
\]
The so-obtained distributions can then be modified to become distributions on $\N^d_*$ with the same regular-variation property. We omit the details.
\end{example}

\begin{remark}
For our main results to hold, we do not impose any assumption on the spectral measures in Examples~\ref{example:standard} and~\ref{example:polar}. The only assumption is the non-standard multivariate regular variation with indices $\alpha_1,\dots,\alpha_d\in(0,1)$, and $\alpha_1<1/2$ when $d=1$.
\end{remark}

\subsection{The random field} \label{randomfield}\

We now associate a random field $(X_{\vec j})_{{\vec j}\in\Z^d}$ to the random graph $\G_\mu$. Assume that $\mu \in \D(E,\nu)$ as in the preceding section, with the diagonal matrix $E$ satisfying $q(E)>2$, and let $p\in(0,1)$. We proceed as follows:

First, generate the random directed graph $\G_\mu$ as described in previous sections, which has almost surely infinitely many connected components in this situation. Let $\{\C_i\mid i\ge 1\}$ denote the collection of disjoint components and associate to each component $\C_i$ a random variable $\varepsilon_i$ such that $(\varepsilon_i)_{i\ge 1}$ are i.i.d.\ with distribution given by $\P(\varepsilon_i=1)=p$ and $\P(\varepsilon_i=-1)=1-p$. Finally, for all $\vec j\in\Z^d$, set $X_{\vec j}=\varepsilon_i$ where $i$ is such that $\vec j\in\C_i$. This construction implies that $X_{\vec j}=X_{\vec k}$ as soon as $\vec j$ and $\vec k$ belong to the same component of $\G_\mu$, and they are independent otherwise.
\begin{remark}
As pointed out already in \citep{hammond13power}, the one-dimensional model is an example of the so-called chains with complete connections, which has a long history with different names; see \citep{fernandez04chains,fernandez05chains} for more references. In the same spirit, our model is an example of partially ordered models recently introduced by \citep{deveaux10partially}, an extension of chains with complete connections to random fields.
\end{remark}

For all ${\vec n}\in\N^d$, we introduce the partial sum
$$
S_{\vec n}=\sum_{{\vec j}\in[{\vec 0},{\vec n}-{\vec 1}]}X_{\vec j}.
$$

Our aim is to establish a functional central limit theorem (invariance principle) for the partial sums $S_{\vec n}$ (with centering and appropriate normalization) when ${\vec n}$ goes to infinity with a specific relative speed in each direction. We will distinguish different regimes.  
We first show, in this section, that  $(X_{\vec j})_{{\vec j}\in\Z^d}$ can be seen as a linear random field with martingale differences innovations, and thus, $S_{\vec n}$ is a partial sum of a linear random field.

\medskip

For all ${\vec j}\in\Z^d$, we define the $\sigma$-fields $\sigma_{\vec j}=\sigma\{X_{\vec k}\mid {\vec k}<{\vec j}\}$ and $\overline\sigma_{\vec j}=\sigma\{X_{\vec k}\mid {\vec k}\ngeq{\vec j}\}$. Note that, for ${\vec j}<{\vec n}$,
the value of $X_{\vec n}$ conditioned on $\sigma_{\vec j}$ is obtained by sampling the ancestral line $A_{\vec n}$ and taking the value of $X_{\vec k}$ where ${\vec k}$ is the first site of the ancestral line $A_{\vec n}$ which is strictly smaller than ${\vec j}$. 
We denote 
\begin{equation}\label{eq:Xj^*}
X^*_{\vec j}=X_{\vec j}-\E(X_{\vec j}\mid \sigma_{\vec j})=X_{\vec j}-\E(X_{\vec j}\mid \overline\sigma_{\vec j}).
\end{equation}
The equality $\E(X_{\vec j}\mid\overline\sigma_{\vec j})=\E(X_{\vec j}\mid\sigma_{{\vec j}})$ comes from the fact that starting from ${\vec j}$, the next site in the ancestral line $A_{\vec j}$ is necessarily strictly smaller than ${\vec j}$.
Then for all ${\vec j}\in\Z^d$, $\E(X^*_{\vec j}\mid \overline\sigma_{\vec j})=0$ and $X^*_{\vec j}$ is measurable with respect to $\overline\sigma_{{\vec j}+{\vec e}_q}$ for all $q=1,\ldots, d$, where ${\vec e}_q$ is the $q$-th canonical unit vector of $\R^d$. In particular, the random variables $X^*_{\vec j}$ are orthogonal to each other, that is, $\E(X^*_{\vec j}X^*_{{\vec k}})=0$ as soon as ${\vec j}\ne{{\vec k}}$.

\begin{lemma}\label{lem:varX^*0}
In the above setting,
$$
\Var(X^*_{\vec 0})=\left(\sum_{{\vec k}\in\N^d}q_{\vec k}^2\right)^{-1}\Var(X_{\vec 0}).
$$
\end{lemma}
\begin{proof} Let ${\vec Z}_{\vec 0}$ be the random variable with distribution $\mu$ that gives the first ancestor of ${\vec 0}$. We have $X_{\vec 0}=\sum_{{\vec k}>{\vec 0}}\ind_{\{{\vec Z}_{\vec 0}={\vec k}\}}X_{-{\vec k}}$ and $\E(X_{\vec 0}|\sigma_{\vec 0})=\sum_{{\vec k}>{\vec 0}}p_{\vec k}X_{-{\vec k}}$, where
 $p_{\vec k}=\mu(\{{\vec k}\})$ for all ${\vec k}>{\vec 0}$.
Therefore,
\begin{align}
\E(X^{*2}_{\vec 0})&=\E\left(\left(\sum_{{\vec k}>{\vec 0}}(\ind_{\{{\vec Z}_{\vec 0}={\vec k}\}}-p_{\vec k})X_{-{\vec k}}\right)^2\right)\nonumber\\
&=\sum_{{\vec k}>{\vec 0}}\sum_{{\vec \ell}>{\vec 0}}\E((\ind_{\{{\vec Z}_{\vec 0}={\vec k}\}}-p_{\vec k})(\ind_{\{{\vec Z}_{\vec 0}={\vec \ell}\}}-p_{\vec \ell}))\E(X_{-{\vec k}}X_{-{\vec \ell}}).\label{b0}
\end{align}
But, 
\begin{equation}\label{b1}
\E(X_{-{\vec k}}X_{-{\vec \ell}})=\P(A_{-{\vec k}}\cap A_{-{\vec \ell}}\ne\emptyset)\E(X_{\vec 0}^2)+\P(A_{-{\vec k}}\cap A_{-{\vec \ell}}=\emptyset)\E(X_{\vec 0})^2.
\end{equation}
and 
\begin{equation}\label{b2}
\E((\ind_{\{{\vec Z}_{\vec 0}={\vec k}\}}-p_{\vec k})(\ind_{\{{\vec Z}_{\vec 0}={\vec \ell}\}}-p_{\vec \ell}))=\ind_{\{{\vec k}={\vec \ell}\}}p_{\vec k}-p_{\vec k}p_{\vec \ell}.
\end{equation}
Combining \eqref{b0}, \eqref{b1}, and \eqref{b2}, we get
\begin{align*}
\E &(X^{*2}_{\vec 0})\\
&=\E(X_{\vec 0}^2)\left(1-\sum_{{\vec k}>{\vec 0}}\sum_{{\vec \ell}>{\vec 0}}p_{\vec k}p_{\vec \ell}\P(A_{-{\vec k}}\cap A_{-{\vec \ell}}\ne\emptyset)\right)
-\sum_{{\vec k}>{\vec 0}}\sum_{{\vec \ell}>{\vec 0}}p_{\vec k}p_{\vec \ell}\P(A_{-{\vec k}}\cap A_{-{\vec \ell}}=\emptyset)\E(X_{\vec 0})^2\\
&=(\E(X_{\vec 0}^2)-\E(X_{\vec 0})^2)\sum_{{\vec k}>{\vec 0}}\sum_{{\vec \ell}>{\vec 0}}p_{\vec k}p_{\vec \ell}\P(A_{-{\vec k}}\cap A_{-{\vec \ell}}=\emptyset)\\
&=\Var(X_{\vec 0})\P(A_{\vec 0}\cap A_{\vec 0}'=\{{\vec 0}\}),
\end{align*}
where $A_{\vec 0}'$ is an independent copy of $A_{\vec 0}$.
Finally, as we saw in \eqref{eq:EA0A0'} in the proof of Lemma~\ref{lem:probainter}, $\sum_{{\vec k}\in\N^d}q_{\vec k}^2=\E|A_{\vec 0}\cap A_{\vec 0}'|=\P(A_{\vec 0}\cap A_{\vec 0}'=\{{\vec 0}\})^{-1}$ and the proof is complete.
\end{proof}

Now, for all ${\vec j}\in\Z^d$, we introduce
$$
\Delta_{\vec j}(X)=\sum_{{\vec \varepsilon}\in\{0,1\}^d}(-1)^{d-|{\vec \varepsilon}|_1}\E(X\mid \sigma_{{\vec j}+{\vec \varepsilon}}),
$$
where $|\vec \varepsilon|_1=\varepsilon_1+\ldots+\varepsilon_d$.

Remark that, since $\E(X_{\vec j}\mid \sigma_{{\vec j}+{\vec \varepsilon}})=\E(X_{\vec j}\mid \sigma_{{\vec j}})$ for all ${\vec \varepsilon}\in\{0,1\}^d$ with the exception of ${\vec \varepsilon}={\vec 1}$ for which
$\E(X_{\vec j}\mid \sigma_{{\vec j}+{\vec 1}})=X_{\vec j}$, we have
\begin{equation}\label{Delta_j}
\Delta_{\vec j}(X_{\vec j})=X_{\vec j}-\E(X_{\vec j}\mid \sigma_{\vec j})=X^*_{\vec j}.
\end{equation}
More generally, we have the following lemma.
\begin{lemma}\label{lem:Delta}
 For all ${\vec j}$, ${\vec k}\in\Z^d$,
$$
\Delta_{\vec j}(X_{\vec k})=q_{{\vec k}-{\vec j}}X^*_{\vec j},
$$
which vanishes when ${\vec k}\ngtr{\vec j}$.
\end{lemma}
\begin{proof}
The result is clear when $\vec k=\vec j$ (see \eqref{Delta_j}). 
In the case ${\vec k}\le{\vec j}$, $\vec k\ne\vec j$,  we easily see that $\Delta_{\vec j}(X_{\vec k})=0$.

Now, assume ${\vec k}\nleq{\vec j}$. 
By linearity, we have
$$
\Delta_{\vec j}(X_{\vec k})=\Delta_{\vec j}(X_{\vec k}\ind_{\{{\vec j}\in A_{\vec k}\}})+\Delta_{\vec j}(X_{\vec k}\ind_{\{{\vec j}\notin A_{\vec k}\}}).
$$
Using first that $X_{\vec k}\ind_{\{{\vec j}\in A_{\vec k}\}}=X_{\vec j}\ind_{\{{\vec j}\in A_{\vec k}\}}$, and then that $\{{\vec j}\in A_{\vec k}\}$ is independent of $\sigma_{\vec j+\vec 1}$, we obtain
$$
\Delta_{\vec j}(X_{\vec k}\ind_{\{{\vec j}\in A_{\vec k}\}})=\Delta_{\vec j}(X_{\vec j})\P({\vec j}\in A_{\vec k})=q_{{\vec k}-{\vec j}}X^*_{\vec j}.
$$ 
Denote by $\vec a(\vec k,\vec j)$ the first element of the ancestral line $A_{\vec k}$ that is $\le \vec j$ and remark that $\vec a(\vec k,\vec j)$ is independent of $\sigma_{\vec j + \vec 1}$.
Then,
$$
\Delta_{\vec j}(X_{\vec k}\ind_{\{{\vec j}\notin A_{\vec k}\}})=\sum_{\vec \ell\le \vec j, \vec \ell\ne \vec j}\Delta_{\vec j}(X_{\vec k}\ind_{\{\vec a(\vec k,\vec j)=\vec \ell\}})
=\sum_{\vec \ell\le \vec j, \vec \ell\ne \vec j}\Delta_{\vec j}(X_{\vec \ell})\P(\vec a(\vec k,\vec j)=\vec \ell).
$$
But, $\Delta_{\vec j}(X_{\vec \ell})=0$ for all $\vec \ell\le \vec j$, $\vec \ell\ne \vec j$,
and we finally have
$$
\Delta_{\vec j}(X_{\vec k}\ind_{\{{\vec j}\notin A_{\vec k}\}})=0,
$$
which completes the proof.
\end{proof}

\begin{lemma}\label{lem:linproc}
For all  ${\vec k}\in\Z^d$, the series $\sum_{{\vec j}\in\Z^d}\Delta_{\vec j}(X_{\vec k})$ converges in $L^2$ and
\begin{equation*}
X_{\vec k}-\E(X_{\vec k})=\sum_{{\vec j}\in\Z^d}\Delta_{\vec j}(X_{\vec k}).
\end{equation*}
\end{lemma}

\begin{proof}
First, remark that by stationarity we may only consider the case where ${\vec k}={\vec 0}$. The sum in the statement can be write as $\sum_{{\vec j}\in\N^d}\Delta_{-\vec j}(X_{\vec 0})$ since the other terms vanish. We denote by ${n\vec 1}$ the vector $(n,\ldots,n)$ where $n\in\N$. By Lemma~\ref{lem:Delta}, we have
$$
\E\left(\left(\sum_{{\vec j}\in[{\vec 0},{n\vec1}]}\Delta_{-{\vec j}}(X_{\vec 0})\right)^2\right)=\E(X^{*2}_{\vec 0})\left(\sum_{{\vec j}\in[{\vec 0},{n\vec1}]}q_{\vec j}^2\right)
$$
and the right hand side converges to $\Var(X_{\vec 0})$ as $n\to\infty$ thanks to Lemma~\ref{lem:varX^*0}.
Now, by construction, the random variables $ \sum_{{\vec j}\in[{\vec 0},{n\vec1}]}\Delta_{-{\vec j}}(X_{\vec 0})$ and $X_{\vec 0}-\sum_{{\vec j}\in[{\vec 0},{n\vec1}]}\Delta_{-{\vec j}}(X_{\vec 0})$
are orthogonal. 
To see this last fact, note that for all $\vec l\le \vec 0$ and $\vec j \le \vec 0$,
$
\E\left( \E( X_{\vec 0} \mid \sigma_{\vec l} ) \mid \sigma_{\vec j} \right) = \E\left( X_{\vec 0} \mid \sigma_{\min\{\vec l, \vec j\}}\right),
$
where the minimum is taken on each coordinate.
Thus, we get
$$
\E\left(\left(X_{\vec 0}-\E(X_{\vec 0})-\sum_{{\vec j}\in[{\vec 0},{n\vec1}]}\Delta_{-{\vec j}}(X_{\vec 0})\right)^2\right)
=\Var(X_{\vec 0})-\E\left(\left(\sum_{{\vec j}\in[{\vec 0},{n\vec1}]}\Delta_{-{\vec j}}(X_{\vec 0})\right)^2\right)\to 0,
$$
as $n\to \infty$.
\end{proof}

From Lemma~\ref{lem:linproc} and Lemma~\ref{lem:Delta}, we get that $(X_{\vec j}-\E(X_{\vec j}))_{{\vec j}\in\Z^d}$ is the linear random field given by the innovations $(X^*_{\vec j})_{{\vec j}\in\Z^d}$ and the filter $(q_{\vec j})_{{\vec j}\in\Z^d}$. That is, for all $\vec k\in\Z^d$,
$$
X_{\vec k}-\E(X_{\vec k})=\sum_{{\vec j}\in\Z^d}q_{{\vec k}-{\vec j}}X^*_{\vec j}.
$$

Hence, we proved the following proposition.
\begin{proposition}\label{prop:linearProc}
For all $\vec n \in \N^d$,
\begin{equation*}
S_{\vec n}-\E(S_{\vec n})=\sum_{{\vec j}\in\Z^d}b_{{\vec n},{\vec j}}X^*_{\vec j}.
\end{equation*}
where $b_{{\vec n},{\vec j}}=\sum_{{\vec k}\in[{\vec 0},{\vec n}-{\vec 1}]}q_{{\vec k}-{\vec j}}$. Further, for any ${\vec n}\in\N^d$, $b_{\vec n}=(b_{{\vec n},{\vec j}})_{{\vec j}\in\Z^d}$ belongs to $\ell^2(\Z^d)$, that is $\|b_{{\vec n}}\|^2:=\sum_{{\vec j}\in\Z^d}b_{{\vec n},{\vec j}}^2<\infty$.
\end{proposition}

\section{A central limit theorem}\label{sec:CLT}
We still assume $\mu\in \D(E,\nu)$, where $\nu$ is a full $E$-operator stable law on $\R_+^d$ with $E=\mbox{diag}(1/\alpha_1,\ldots,1/\alpha_d)$, with $\alpha_i\in(0,1)$ and $\alpha_1\in(0,1/2)$ if $d=1$. The random field
$(X_{\vec j})_{\vec j\in \Z^d}$ is the random field defined in Section~\ref{randomfield}.
In view of Proposition~\ref{prop:linearProc}, we want to establish central limit theorems for the sequences of $L^2$ random variables
$$
\sum_{{\vec j}\in\Z^d}c_{n,{\vec j}}X^*_{\vec j},\qquad n\ge 1
$$  
with general coefficients $c_n=(c_{n,\vec j})_{\vec j\in\Z^d}\in \ell^2(\Z^d)$. Recall the definition of $X^*_{\vec j}$ in \eqref{eq:Xj^*}.
It turns out that a simple assumption on $c_n$ for a central limit theorem is given by
\begin{equation}\label{eq:sup}
  \lim_{n\to\infty}\sup_{\vec j\in\Z^d}\frac{|c_{n,\vec j}|}{\|c_n\|}=0.
\end{equation}
The aim of this section is to prove the following central limit theorem.
\begin{theorem}\label{thm:clt}
 Let $c_n=(c_{n,\vec j})_{\vec j\in\Z^d}$ be a sequence in $\ell^2(\Z^d)$ satisfying \eqref{eq:sup}.
Then 
\[
 \frac{1}{\|c_n\|}\sum_{{\vec j}\in \Z^d}c_{n,\vec j}X^*_{\vec j} \weakto \cN(0,\sigma^2_X) \mbox{ as }n\to\infty,
\]
where 
\equh\label{eq:sigma}
\sigma_X^2 := \Var(X_{\vec 0}^*)= \frac{\Var(X_{\vec 0})}{\sum_{{\vec k}\in\N^d}q_{\vec k}^2}.
\eque
\end{theorem}
\begin{proof}
Recall that we write $\sigma_{\vec j}=\sigma\{X_{\vec k}\mid \vec k<\vec j\}$ and $\overline\sigma_{\vec j}=\sigma\{X_{\vec k}\mid \vec k\ngeq\vec j\}$, and we already have seen for all $\vec j\in\Z^d$,
$$
\E(X_{\vec j}\mid \sigma_{\vec j})=\E(X_{\vec j}\mid \overline\sigma_{\vec j}).
$$
We now consider the $\sigma$-fields $\F_{\vec j}=\sigma\{X_{\vec k}\mid \vec k \prec \vec j\}$, where $\prec$ denotes the lexicographical order on $\Z^d$. We have  $\sigma_{\vec j}\subset \F_{\vec j}\subset \overline\sigma_{\vec j}$ for all $\vec j\in\Z^d$ and thus, for all $\vec j\in\Z^d$, we also have
$$
\E(X_{\vec j}\mid \F_{\vec j})=\E(X_{\vec j}\mid \sigma_{\vec j}).
$$

In general, if $\{\F_{\vec i}\}_{{\vec i}\in\Z^d}$ is a filtration such that $\F_{\vec i}\subset\F_{\vec j}$ if ${\vec i}\prec{\vec j}$, for all ${\vec i},{\vec j}\in\Z^d$, we say that integrable random variables $(\xi_{\vec i})_{{\vec i}\in\Z^d}$ are martingale differences with respect to $\{\F_{\vec i}\}_{{\vec i}\in\Z^d}$ if 
$$
\xi_{\vec i}\in\F_{\vec i + {\vec e}_d}\;\mbox{ and }\;\E(\xi_{\vec i}\mid\F_{{\vec i}}) = 0\;\mbox{ for all }{\vec i}\in\Z^d,
$$
where ${\vec e}_d$ is the $d$-th vector of the canonical basis of $\R^d$. 

Thus, by definition (see \eqref{eq:Xj^*}), the random field $(X_{\vec j}^*)_{\vec j\in\Z^d}$ is composed of martingale differences with respect to the filtration $\{\F_{\vec j}\}_{{\vec j}\in\Z^d}$ defined above. As a consequence we will be able to use the following theorem which is an obvious adaptation of a theorem of \citet{mcleish74dependent} for triangular array of $\Z$-indexed martingale differences.
\begin{theorem}[\citet{mcleish74dependent}]\label{thm:McLeish}
Let $(\xi_{n, \vec j})_{n\in\N, \vec j\in\Z^d}$ be a collection of random variables satisfying
$\sum_{\vec j\in\Z^d}\xi_{n, \vec j}\in L^2$ for all $n\in\N$.
Assume that for each $n\in\N$, $(\xi_{n, \vec j})_{\vec j\in\Z^d}$ are martingale differences with respect to a filtration $\{\F_{n,\vec j}\}_{\vec j\in\Z^d}$ in the lexicographical order.
If
\begin{enumerate}[(i)]
 \item\label{McL1} $\lim_{n\to\infty}\max_{\vec j\in\Z^d}|\xi_{n, \vec j}| = 0$ in probability,

\item \label{McL2} $\sup_{n\in\N}\E\left(\max_{\vec j\in\Z^d}\xi_{n, \vec j}^2\right)<\infty$,

\item \label{McL3} $\lim_{n\to\infty}\sum_{\vec j\in\Z^d}\xi_{n, \vec j}^2 = \sigma^2>0$ in probability,
\end{enumerate}
\noindent then 
$$\sum_{\vec j\in\Z^d}\xi_{n, \vec j}\weakto \cN(0,\sigma^2) \mbox{ as }n\to\infty.$$
\end{theorem}
\begin{proof}
Let us explain how one can derive this theorem from Theorem 2.3 in \citep{mcleish74dependent} which is stated for finite sets of random variables at each $n$. First, 
since $\sum_{\vec j\in\Z^d}\xi_{n, \vec j}\in L^2$, one can find a sequence of finite rectangles $\Gamma_n$ in $\Z^d$ such that $\sum_{\vec j\in\Z^d\setminus \Gamma_n}\xi_{n, \vec j}$ converges to $0$ in $L^2$ as $n\to\infty$. Thus, the conclusion of Theorem~\ref{thm:McLeish} holds as soon as 
 $$\sum_{\vec j\in\Gamma_n}\xi_{n, \vec j}\weakto \cN(0,\sigma^2) \mbox{ as }n\to\infty.$$
Furthermore, for each $n$, using the lexicographical order on the finite set $\Gamma_n$, one can re-index the random variables $(\xi_{n,\vec j})_{\vec j \in \Gamma_n}$ and the $\sigma$-fields $\{\F_{n,\vec j}\}_{\vec j\in \Gamma_n}$ in order to fit with the statement of \citep[Theorem 2.3]{mcleish74dependent}. Now, it suffices to observe that conditions \iteqref{McL1}, \iteqref{McL2}, and \iteqref{McL3} imply those of  \citep[Theorem 2.3]{mcleish74dependent}.
\end{proof}
Theorem~\ref{thm:clt} will be established by application of Theorem~\ref{thm:McLeish} to 
$$
\xi_{n, \vec j}:= \frac{c_{n,{\vec j}}}{\|c_n\|}X^*_{\vec j} \mbox{ and }
\F_{n,\vec j}:=\F_{\vec j}=\sigma\{X_{\vec k}\mid \vec k \prec \vec j\}.
$$
Note that $|X_{\vec j}^*|\le 2$ and thus the conditions \iteqref{McL1} and \iteqref{McL2} can be reduced to a condition on the coefficients $c_{n,\vec j}$. Indeed, \iteqref{McL1} and \iteqref{McL2} are satisfied as soon as \eqref{eq:sup} holds.
Condition \iteqref{McL3} can be derived from the following lemma.
\begin{lemma}\label{lem:ergodicity}
Let $c_n=(c_{n,\vec j})_{\vec j\in\Z^d}$ be a sequence in $\ell^2(\Z^d)$ such that \eqref{eq:sup} holds. Then, 
\begin{equation*}
\lim_{{n}\to\infty}\frac{1}{\|c_n\|^2}\sum_{{\vec j}\in \Z^d}c_{n,\vec j}^2X^{*2}_{{\vec j}} = \E(X^{*2}_{{\vec 0}}) \mbox{ in } L^2,
\end{equation*}
\end{lemma}
\begin{proof}
We start by showing that 
\begin{equation}\label{eq:covX*j}
\Cov(X^{*2}_{\vec i},X^{*2}_{\vec j})\to 0 ,\mbox{ as } |{\vec i}-{\vec j}|_\infty\to\infty.
\end{equation}
Observe that $X^{*}_{{\vec j}}=X_{\vec j}-\sum_{{\vec \ell}>{\vec 0}}p_{\vec \ell}X_{{\vec j}-{\vec \ell}}$ and let $X^{*}_{{\vec j},k}=X_{\vec j}-\sum_{{\vec \ell}\in \{1,\ldots ,k\}^d}p_{\vec \ell}X_{{\vec j}-{\vec \ell}}$.
For any ${\vec j}\in\Z^d$, using that $|X_{\vec j}^*|\le 2$, we get
\[
\left|X^{*2}_{\vec j}-X^{*2}_{{\vec j},k}\right|
\le 4\left|X^{*}_{\vec j}-X^{*}_{{\vec j},k}\right|
=4\left|\sum_{{\vec \ell}\in [1,\infty)^d\setminus[1,k]^d}p_{\vec \ell}X_{{\vec j}-{\vec \ell}}\right|
\]
Thus, since $|X_{\vec j}|=1$ for all ${\vec j}\in\Z^d$, 
\begin{equation}
\sup_{{\vec j}\in\Z^d} \left|X^{*2}_{\vec j}-X^{*2}_{{\vec j},k}\right|\le 4\, \mu\left([1,\infty)^d\setminus[1,k]^d\right) \mbox{ a.s., for all }k>0.\label{eq:erg:2}
\end{equation}
Now, introduce
$$
R_{{\vec i},{\vec j},k}=\left\{\left(\bigcup_{{\vec \ell}\in{\vec i}-[0,k]^d}A_{\vec \ell}\right)\cap\left( \bigcup_{{\vec m}\in{\vec j}-[0,k]^d}A_{\vec m} \right)=\emptyset\right\}.
$$
We have
$$
\P(R_{{\vec i},{\vec j},k}^c)\le \sum_{{\vec \ell}\in{\vec i}-[0,k]^d}\sum_{{\vec m}\in{\vec j}-[0,k]^d} \P(A_{\vec \ell}\cap A_{\vec m}\ne \emptyset).
$$
But, from Lemma~\ref{lem:probainter} \iteqref{lem:probainter:1}, we see that $\P(A_{\vec \ell}\cap A_{\vec m}\ne \emptyset)\to 0$ as $|{\vec \ell}-{\vec m}|_\infty\to \infty$ and thus, for any $k\ge 1$,
\begin{equation}\label{eq:erg:3}
\P(R_{{\vec i},{\vec j},k}^c)\to 0, \mbox{ as }|{\vec i}-{\vec j}|_\infty\to\infty.
\end{equation}
Fix $\varepsilon>0$ and, using \eqref{eq:erg:2}, let $k\in\N$ be such that $\sup_{{\vec j}\in\Z^d} |X^{*2}_{\vec j}-X^{*2}_{{\vec j},k}|<\varepsilon$. From \eqref{eq:erg:3}, for $|{\vec i}-{\vec j}|_\infty $ large enough, we have $\P(R_{{\vec i},{\vec j},k}^c)<\varepsilon$ and we obtain
\begin{align*}
 \E(X^{*2}_{\vec i}X^{*2}_{\vec j})
&= \E(X^{*2}_{{\vec i},k}X^{*2}_{{\vec j},k}) +  O(\varepsilon)= \E(X^{*2}_{{\vec i},k}X^{*2}_{{\vec j},k}\mid R_{{\vec i},{\vec j},k}) + O(\varepsilon) \\
&= \E(X^{*2}_{{\vec i},k} \mid R_{{\vec i},{\vec j},k})\E(X^{*2}_{{\vec j},k}\mid R_{{\vec i},{\vec j},k}) + O(\varepsilon) = \E(X^{*2}_{{\vec i},k})\E(X^{*2}_{{\vec j},k}) +  O(\varepsilon) \\
&= \E(X^{*2}_{{\vec i}})\E(X^{*2}_{{\vec j}}) +  O(\varepsilon).
\end{align*}
This proves \eqref{eq:covX*j}.

To prove the lemma, fix $\varepsilon>0$ and let $K>0$ be such that $|\Cov(X_{\vec j}^{*2},X^{*2}_{\vec i})|\le \varepsilon$ as soon as $|{\vec i}-{\vec j}|_\infty>K$. One has, 
\begin{align*}
&\E\left(\frac{1}{\|c_n\|^2}\sum_{{\vec j}\in \Z^d}c_{n,\vec j}^2X^{*2}_{{\vec j}}-\E(X_{\vec 0}^{*2})\right)^2\\
&=\sum_{{\vec j}\in \Z^d}\frac{c_{n,\vec j}^2}{\|c_n\|^2}\sum_{{\vec i}\in \Z^d}\frac{c_{n,\vec i}^2}{\|c_n\|^2}\Cov(X_{\vec j}^{*2},X^{*2}_{\vec i})\\
&\le\sum_{{\vec j}\in \Z^d} \frac{c_{n,\vec j}^2}{\|c_n\|^2} \sum_{|{\vec i}-{\vec j}|_\infty\le K}\frac{c_{n,\vec i}^2}{\|c_n\|^2}|\Cov(X_{\vec j}^{*2},X^{*2}_{\vec i})|
+ \varepsilon \sum_{{\vec j}\in \Z^d}\frac{c_{n,\vec j}^2}{\|c_n\|^2} \sum_{|{\vec i}-{\vec j}|_\infty> K}\frac{c_{n,\vec i}^2}{\|c_n\|^2}\\
&\le\sup_{\vec k\in\Z^d}\frac{c_{n,\vec k}^2}{\|c_n\|^2} \sum_{|{\vec i}-{\vec 0}|_\infty\le K}|\Cov(X_{\vec 0}^{*2},X^{*2}_{\vec i})|
+ \varepsilon,
\end{align*}
and the first term of the right hand side goes to $0$ as $n\to\infty$ because $|\Cov(X_{\vec 0}^{*2},X^{*2}_{\vec i})|$ is bounded and $\sup_{\vec k\in\Z^d}{c_{n,\vec k}^2}=o({\|c_n\|^2})$ by assumption.
\end{proof}
Thus the conditions \iteqref{McL1}, \iteqref{McL2}, and \iteqref{McL3} are satisfied with $\sigma_X^2=\E(X_{\vec 0}^{*2})=\Var(X_{\vec 0}^*)=(\sum_{{\vec k}\in\N^d}q_{\vec k}^2)^{-1}\Var(X_{\vec 0})$ by Lemma \ref{lem:varX^*0}.
  Then Theorem~\ref{thm:McLeish} applies and using Lemma~\ref{lem:varX^*0} we complete the proof of Theorem~\ref{thm:clt}.
\end{proof}

The following lemma gives another useful condition on the coefficients $(c_{n,\vec j})_{\vec j\in\Z^d}$ for Theorem~\ref{thm:clt}.

\begin{lemma}\label{lem:trans-sup}
If $(c_{n,\vec j})_{\vec j\in\Z^d}$ is a sequence in $\ell^2(\Z^d)$ that satisfies, for  all $q=1,\dots,d$,
\begin{equation}\label{eq:trans}
\lim_{{n}\to\infty}\frac1{\|c_{n}\|^2}\sum_{{\vec j}\in\Z^d}|{c_{{n},{\vec j}}^2 - c_{{n}, {\vec j}+{\vec e}_q}^2}| = 0,
\end{equation}
where $\vec e_q$ is the $q$-th vector of the canonical basis of $\R^d$,
then \eqref{eq:sup} holds. 
\end{lemma}
\begin{proof}
We use an idea of \citep{peligrad97central}.
Assume that \eqref{eq:sup} does not hold. Then, there exist $\varepsilon>0$, a sequence $(n_k)_{k\ge1}$ such that $n_k\to\infty$ as $k\to\infty$, and a sequence $(\vec j_k)_{k\ge1}$ such that
$c_{n_k,\vec j_k}>\varepsilon \|c_{n_k}\|$ for all $k\in\N$.
Choose $M>0$ such that $M^d\varepsilon^2>1$. One has, for all $k\in\N$,
$$
\|c_{n_k}\|^2 \ge \sum_{\vec j \in[0,M-1]^d}c_{n_k,\vec j_k+\vec j}^2   \ge M^d c_{n_k,\vec j_k}^2 - \sum_{\vec j \in[0,M-1]^d}| c_{n_k,\vec j_k+\vec j}^2 -  c_{n_k,\vec j_k}^2|.
$$
Hence,
\begin{equation}\label{eq:contradic}
(M^d\varepsilon^2-1) \|c_{n_k}\|^2\le  \sum_{\vec j \in[0,M-1]^d}| c_{n_k,\vec j_k+\vec j}^2 -  c_{n_k,\vec j_k}^2|.
\end{equation}
But, if $\vec j \in[0,M-1]^d$, then 
$$
|c_{n_k,\vec j_k}^2 -  c_{n_k,\vec j_k+\vec j}^2|\le \sum_{i=1}^{\ell(\vec \lambda)}  |c_{n_k,\vec \lambda_i}^2 -  c_{n_k,\vec \lambda_{i+1}}^2|,
$$
where $\vec \lambda=(\vec \lambda_0, \vec \lambda_1,\ldots, \vec \lambda_{\ell})$ is any path from $\vec \lambda_0=\vec j_k$ to $\vec \lambda_\ell=\vec j_k + \vec j$, with $|\vec \lambda_i-\vec \lambda_{i+1}|_1=1$. Since $\vec j \in[0,M-1]^d$, we can always choose the path $\vec \lambda$ of length $\ell=\ell(\vec \lambda)$ smaller than $dM$. Thus, we get
\begin{align*}
|c_{n_k,\vec j_k}^2 -  c_{n_k,\vec j_k+\vec j}^2|&\le dM \sup_{q=1,\ldots,d} \sup_{\vec k\in\Z^d}|c_{n_k,\vec k}^2 -  c_{n_k,\vec k+\vec e_q}^2|\\
&\le dM \sum_{q=1}^d \sum_{\vec k\in\Z^d}|c_{n_k,\vec k}^2 -  c_{n_k,\vec k+\vec e_q}^2|.
\end{align*}
Together with \eqref{eq:contradic}, this contradicts  \eqref{eq:trans}.
\end{proof}

\begin{remark}
 Using Cauchy-Schwarz inequality, we also see that the condition 
\begin{equation}\label{eq:trans2}
\lim_{{n}\to\infty}\frac1{\|c_{n}\|^2}\sum_{{\vec j}\in\Z^d}({c_{{n},{\vec j}} - c_{{n}, {\vec j}+{\vec e}_q}})^2 = 0, \mbox{ for  all }q=1,\dots,d,
\end{equation}
implies \eqref{eq:trans} and thus by Lemma~\ref{lem:trans-sup}, implies \eqref{eq:sup}. This last observation leads to an improvement in Theorem 3.1 in \citet{bierme14invariance}. The conditions (i) and (ii) of this theorem are equivalent to our conditions \eqref{eq:sup} and \eqref{eq:trans2}, respectively. Thus, the condition (i) in \cite[Theorem 3.1]{bierme14invariance} is unnecessarily.

\end{remark}

\section{An invariance principle}\label{sec:WIP}
 \label{main}

The aim of the section is to establish a general invariance principle for partial sums of the random field $(X_{\vec j})_{\vec j\in\Z^d}$ defined in Section~\ref{sec:model}. Recall that $(X_{\vec j})_{\vec j\in\Z^d}$ are associated to the random graph $\mathcal G_\mu$, with $\mu\in\D(E,\nu)$. We consider partial sums on finite rectangular subsets of $\Z^d$. As we will see, the growth of the rectangles will be determinant in the invariance principle and different limit random fields appear at different regimes.
For the general case, consider a matrix $E'=\mbox{diag}(1/\alpha'_1,\ldots,1/\alpha'_d)$ with $\alpha_i'>0, i=1,\dots,d$ and the partial-sum process
$$
S^{E'}_{n}({\vec t})=\sum_{{\vec j}\in[\vec 0,n^{E'}\vec t-\vec 1]}X_{\vec j},\qquad n\ge 1 \mbox{ and }\vec t=(t_1,\ldots,t_d)\in[0,1]^d.
$$
The result will depend on both $E'$ and $E$. 

We introduce several parameters. First, for all $k=1,\ldots,d$, set $\rho_k:=\alpha_k / \alpha'_k$, and
consider
\equh\label{def:gamma0}
\gamma_0=\gamma_0(E,E'):=\min\ccbb{\gamma\in\{\rho_1,\ldots,\rho_d\}\mmid \sum_{k:\gamma\geq \rho_k}\frac1{\alpha_k}>2, \sum_{k:\gamma>\rho_k}\frac1{\alpha_k}\leq 2}.
\eque
Note that $\gamma_0$ is well defined by the assumption $q(E)>2$, and is completely determined by $E$ and $E'$.
Given $\gamma_0>0$, define the sets
\begin{align*}
I_<&:=\{k\in\{1,\ldots,d\}\mid \gamma_0<\rho_k\},\\ 
I_=&:=\{k\in\{1,\ldots,d\}\mid \gamma_0=\rho_k\},\\
I_>&:=\{k\in\{1,\ldots,d\}\mid \gamma_0>\rho_k\}.
\end{align*}
This gives  a partition of $\{1,\ldots,d\}$. We also write $I_\le:=I_<\cup I_=$ and  $I_\ge:=I_=\cup I_>$.
The sets $I_>$ and $I_<$ consist of the directions in which the limit random field exhibit degenerate dependence structure. Remark that by construction, 
\[
|I_=|\geq 1 \quad\mbox{ and }\quad |I_>|\leq 1.
\] 
According to these subsets of $\{1,\dots,d\}$, we consider subspaces of $\R^d$ given by
\begin{align*}
\H_<&:=\{x\in\R^d\mid x_k=0 \mbox{ for }k\notin I_<\}, 
\end{align*}
and similarly $\H_=,\H_>,\H_\leq,\H_\geq$. Let $\pi_<$, $\pi_=$, $\pi_>$,  $\pi_\le$, and  $\pi_\ge$ denote orthogonal projections to the corresponding subspaces, and let $\lambda_<, \lambda_=, \lambda_>, \lambda_\le$, and $\lambda_\ge$ denote the Lebesgue measures on the corresponding subspaces.
For $\pi$ of any proceeding projection, $\pi E$ is a linear operator on $\Rd$; accordingly there is  a corresponding diagonal matrix, which we also denote by $\pi E$ with a little abuse of notations.

Next, we define another diagonal matrix $E''$ (that only depends on $E$ and $E'$) by:
\begin{equation}\label{def:E''}
 E'':=\mbox{diag}(\gamma_1/\alpha'_1,\ldots,\gamma_d/\alpha'_d), \;\mbox{ with }\gamma_k:=\frac{\gamma_0}{\rho_k}\vee 1,\; k=1,\ldots,d.
\end{equation}
Remark that, by definition of $E''$, one has
$ \pi_\le E''=\pi_\le E'$ and $ \pi_\ge E''= \gamma_0\pi_\ge E$.
Further, $E''-\gamma_0 E$ is strictly positive on $\H_<$.

We can now state our main result.
\begin{theorem}\label{thm:main}
Assume $\mu\in\D(E,\nu)$ with $E = {\rm diag}(1/\alpha_1,\dots,1/\alpha_d)$ with $\alpha_i\in(0,1), i=1,\dots, d$, and $\alpha_1\in(0,1/2)$ if $d=1$. Let $E'={\rm diag}(1/\alpha'_1,\ldots,1/\alpha'_d)$, with $\alpha'_i>0, i=1,\dots,d$, and $\gamma_0$ defined as in \eqref{def:gamma0}. If $q(\pi_> E)< 2$, then 
\[
\left(\frac{S^{E'}_{n}({\vec t})-\E(S^{E'}_{n}({\vec t}))}{n^{\gamma_0+q(E')-q(E'')/2}}\right)_{\vec t \in [0,1]^d}\dconv (W(\vec t))_{\vec t \in [0,1]^d},
\]
as $n\to\infty$, in the Skorohod space $D([0,1]^d)$,
where $(W(\vec t))_{\vec t\in\R^d}$ is a zero-mean Gaussian process with covariances given by
\begin{multline*}
\Cov(W(\vec t),W(\vec s))=\sigma_X^2
\left(\prod_{k\in I_<}\Cov(B_{1/2}(t_k),B_{1/2}(s_k))\right)\\
\times\left( \prod_{k\in I_>}\frac{t_ks_k}{2\pi} \right) \int_{\H_\ge} |\log \psi(\vec y)|^{-2} \prod_{k \in I_=} \frac{(e^{it_k y_k}-1)\overline{(e^{is_k y_k}-1)}}{2\pi|y_k|^2}\, d\lambda_\ge(\vec y),
\end{multline*}
with $B_{1/2}$ a standard Brownian motion on $\R$, $\psi$ is the characteristic function of $\nu$, and $\sigma_X^2$ is given in~\eqref{eq:sigma}.
\end{theorem}
This theorem reveals that taking different summing rectangles may lead to different limits, under different normalizations. To the best of our knowledge, such a phenomenon has not been noticed in the literature until very recently \citep{puplinskaite13aggregation,puplinskaite15scaling} for a different model. We elaborate more this phenomenon of scaling transition in Section~\ref{sec:property}. 

\begin{remark}
Observe that one can write $\Cov(W(\vec t),W(\vec s)) = \frac{\sigma_X^2}{(2\pi)^{d}}\wb C(\vec t,\vec s)$ with
\begin{equation}\label{eq:C}
\wb C(\vec t, \vec s):= \left(\prod_{k\in I_>}t_ks_k\right)
\int_{\R^d} \left|\log\psi(\pi_\ge\vec y)\right|^{-2} \left(\prod_{k\in I_\le}\frac{\left(e^{it_k y_k}-1\right)\overline{\left(e^{is_k y_k}-1\right)}}{|y_k|^2}\right)d\vec y,
\end{equation}
because of the identity \citep[Proposition 7.2.8]{samorodnitsky94stable}:
\equh\label{eq:covfBm}
\int_\R\frac{(e^{ity}-1)\wb{(e^{isy}-1)}}{2\pi |y|^{1+2H}}dy = C_{H}\Cov\pp{B_{H}(t),B_{H}(s)}, t,s\in\R, H\in(0,1)
\eque
with
\[
C_H = \frac{\pi}{H\Gamma(2H)\sin(H\pi)}.
\]
\end{remark}

The rest of the section is devoted to the proof of Theorem~\ref{thm:main}. 
Using Proposition~\ref{prop:linearProc}, we get
\equh\label{eq:linear2}
S^{E'}_{n}({\vec t})-\E(S^{E'}_{n}({\vec t}))=\sum_{{\vec j} \in \Z^d}b_{n,\vec j}(\vec t)X_{\vec j}^*,
\eque
with $b_n(\vec t)=\left(b_{n,\vec j}(\vec t)\right)_{\vec j \in \Z^d}\in\ell^2(\Z^d)$ and
\begin{equation}\label{b_nj}
b_{n,\vec j}(\vec t)=\sum_{{\vec k}\in[0,n^{E'}\vec t-\vec 1]}q_{\vec k-\vec j}.
\end{equation}
Recall that $(X_{\vec j}^*)_{\vec j\in\Z^d}$ are stationary martingale differences.

The proof of Theorem \ref{thm:main} is now divided into three steps. The key step is to compute the covariance, which is done in Section~\ref{sec:covariance}. Then, we proceed with the standard argument to show the weak convergence by first establishing finite-dimensional convergence in Section~\ref{sec:fdd} and then the tightness in Section~\ref{sec:tightness}. The matrices $E$ and $E'$, and thus $\gamma_0$ and $E''$, are fixed as in the assumptions of the theorem.

\begin{remark}
As we will see below in the proof, essentially we establish an invariance principle for linear random field $(X_{\vec j})_{\vec j\in\Z^d}$ with 
\[
X_{\vec j} = \sum_{\vec i}q_{\vec i}X_{\vec j}^*, \vec j\in \Z^d,
\]
where $(X_{\vec i}^*)_{\vec i\in\Z^d}$ are stationary martingale-difference innovations, and $(q_{\vec i})_{\vec i\in\Z^d}$ are real Fourier coefficients of certain function $Q(\vec t)$. This is a standard framework to obtain linear random fields in the literature, and we comment briefly on connections between our results and others. 
\begin{itemize}
\item[(i)] First, the same invariance principle should hold if the innovations are replaced by other weakly dependent random fields (weakly dependent in the sense of e.g.~\citep{bierme14invariance,lavancier07invariance,wang14invariance}). These results can be viewed as generalizations of the seminal work of \citet{davydov70invariance} on invariance principles for linear processes. 
\item[(ii)]Second, from the modeling point of view, the specific choices of $Q(t)$ (in terms of $\mu\in\D(E,\nu)$) and hence $(q_{\vec j})_{\vec j\in\Z^d}$ are new. However, although our assumption on $Q(\vec t)$ is very general, not all possible operator-scaling Gaussian random fields can show up in the limit; in particular the Hammond--Sheffield model in high dimensions does not scale to fractional Brownian sheets except for a few cases in terms of Hurst indices shown in Proposition~\ref{prop:fBs}. The aforementioned results \citep{bierme14invariance, lavancier07invariance,wang14invariance} all include linear random-field models scaling to fractional Brownian sheets, for flexible choices of Hurst indices.
\item[(iii)] At last, when the innovation random fields exhibit strong dependence, the limiting object could be more complicated (\citep{lavancier07invariance}). 
\end{itemize}
\end{remark}

\subsection{Covariances}\label{sec:covariance}\

From~\eqref{eq:linear2},  we obtain
for $\vec t, \vec s \in [0,1]^d$,
$$
\Cov(S^{E'}_{n}({\vec t}),S^{E'}_{n}({\vec s}))=  \sigma_X^2\langle b_n(\vec t), b_n(\vec s)\rangle,
$$
where, $\langle b_n(\vec t), b_n(\vec s)\rangle:=\sum_{\vec k\in \Z^d}b_{n,\vec k}(\vec t)b_{n,\vec k}(\vec s)$.
The asymptotic behavior of the covariances are given in the following lemma where
$u_n\underset{n\rightarrow \infty}{\sim}v_n$ stands for $\lim_{n\to\infty}u_n/v_n=1$. 
\begin{lemma}\label{lem:cov}
For all $\vec t,\vec s \in [0,1]^d$,
\[
\sigma_X^2\langle b_n(\vec t), b_n(\vec s)\rangle
\underset{n\rightarrow \infty}{\sim} n^{2\gamma_0+2q(E')-q(E'')}\Cov(W(\vec t),W(\vec s)).
\]
\end{lemma}
\begin{proof}

Define
for $m\in\N$ and  $x\in\R$,
\begin{equation}\label{Dirichlet}
D_m(x)=\sum_{l=0}^{m}e^{ilx}=\frac{e^{i(m+1)x}-1}{e^{ix}-1},
\end{equation}
and for $\vec x \in \R^d$, the trigonometric polynomial
$$K_n(\vec t, \vec x)=\sum_{{\vec j}\in\Z^d}\ind_{\vec j\in[0,n^{E'}\vec t-\vec 1]}e^{i\vec j\cdot \vec x}=\prod_{k=1}^dD_{\lfloor n^{1/\alpha'_k}t_k-1\rfloor}(x_k),$$ 
where $\lfloor\cdot\rfloor$ stands for the integer part.
Recall that since
$$Q(\vec x)=\sum_{{\vec j}\in\Z^d}q_{\vec j}e^{i\vec j\cdot \vec x},$$
the sequence $b_{n}(\vec t)$ (defined in \eqref{b_nj}) is obtained by the convolution product of the Fourier coefficients of $K_n(\vec t,\cdot)$
and $\overline{Q}$ with $\overline{Q}(\vec x)=Q(-\vec x)$ since $(q_{\vec j})_{\vec j\in\Z^d}$ is a real sequence. It follows that $b_{n,\vec k}(\vec t)$ is the $\vec k$-th 
Fourier coefficient of $\overline{Q}K_n(\vec t,\cdot)$.
Therefore, using Bessel--Parseval identity, we get
\begin{align}
\langle b_n(\vec t), b_n(\vec s)\rangle&=\frac{1}{(2\pi)^d}\int_{[-\pi,\pi]^d}\overline{Q}(\vec x)K_n(\vec t,\vec x)\overline{\overline{Q}(\vec x)K_n(\vec s,\vec x)}d\vec x\nonumber\\ 
&=\frac{1}{(2\pi)^d}\int_{[-\pi,\pi]^d}\left|Q(\vec x)\right|^2\prod_{k=1}^dD_{\lfloor n^{1/\alpha'_k}t_k-1\rfloor}(x_k)\overline{D_{\lfloor n^{1/\alpha'_k}s_k-1\rfloor}(x_k)}d\vec x\nonumber\\
&=\frac{n^{-q(E'')}}{(2\pi)^d}\int_{n^{E''}[-\pi,\pi]^d} \Phi_n(\vec y,\vec t, \vec s)\,d\vec y,\label{eq:bntbns}
\end{align}
where
\[
 \Phi_n(\vec y,\vec t, \vec s):=\left|Q(n^{-E''}\vec y)\right|^2\prod_{k=1}^dD_{\lfloor n^{1/\alpha'_k}t_k-1\rfloor}(n^{-\gamma_k/\alpha'_k}y_k)\overline{D_{\lfloor n^{1/\alpha'_k}s_k-1\rfloor}(n^{-\gamma_k/\alpha'_k}y_k)}.
\]
According to Lemma~\ref{Qhomogene} and the $E$-homogeneity of $\log\psi$, one has 
\begin{align*}
n^{-2\gamma_0}\left|Q(n^{-E''}\vec y)\right|^2
&=n^{-2\gamma_0}\left|g(n^{-E''}\vec y)\right|^2\left|\log\psi(n^{-\gamma_0E}n^{-(E''-\gamma_0E)}\vec y)\right|^{-2}\\
&=\left|g(n^{-E''}\vec y)\right|^2\left|\log\psi(n^{-(E''-\gamma_0E)}\vec y)\right|^{-2}.
\end{align*}
Thus,
$$
\limn n^{-2\gamma_0}\left|Q(n^{-E''}\vec y)\right|^2 =\left|\log\psi(\pi_\ge\vec y)\right|^{-2},
$$
because $E''-\gamma_0E$ is null on $\H_\ge$ and strictly positive on $\H_<$ and $g(\vec 0)=1$. Further, for all  $n\in\N_*$, $\vec y \in n^{E''}[-\pi,\pi]^d$,
\begin{equation}\label{eq:boundQ}
n^{-2\gamma_0}\left|Q(n^{-E''}\vec y)\right|^2\le \underset{\vec x\in[-\pi,\pi]^d}{\max}|g(\vec x)|^2\sup_{\vec z\in\H_<}|\log\psi(\vec z+\pi_\ge\vec y)|^{-2}.
\end{equation}
Now, remark that for all $t\in[0,1]$ and $y\in\R$,
$$
\lim_{n\to \infty} n^{-1} D_{\lfloor nt-1\rfloor}(n^{-\gamma}y)=\left\{
\begin{array}{ll}
\displaystyle\frac{e^{ity}-1}{iy}&\mbox{ if }\gamma=1\\
\displaystyle t&\mbox{ if }\gamma>1
\end{array}\right.,
$$
and if $|n^{-\gamma}y|\le \pi$, then
$$
\left|n^{-1}D_{\lfloor nt-1\rfloor}(n^{-\gamma}y)\right|
=\left|\frac{\sin\left(\lfloor nt\rfloor n^{-\gamma}y/2\right)}{n\sin\left(n^{-\gamma}y/2\right)} \right|\le\left\{
\begin{array}{ll}
\displaystyle \pi\min\left\{1,\frac{1}{|y|}\right\}&\mbox{ if }\gamma=1\\
\displaystyle \frac{\pi}{2} &\mbox{ if }\gamma>1                                                                                    \end{array}\right.,
$$
where we have used that $\frac{2}{\pi}|x|\le|\sin(x)|\le |x|\wedge1$ for $x\in[-\pi/2,\pi/2]$ and that $|t|\le 1$. Since $\gamma_k>1$ if and only if $k\in I_>$, we infer
\begin{equation}\label{eq:phi_n}
\Phi_n(\vec y,\vec t, \vec s)\sim n^{2\gamma_0+2q(E')} 
\left|\log\psi(\pi_\ge\vec y)\right|^{-2} \left(\prod_{k\in I_>}t_ks_k\right)\left(\prod_{k\in I_\le}\frac{\left(e^{it_k y_k}-1\right)\overline{\left(e^{is_k y_k}-1\right)}}{|y_k|^2}\right)
\end{equation}
as $n\to\infty$ and for all $\vec t$, $\vec s\in [0,1]^d$, 
\equh\label{eq:DCT2}
n^{-2\gamma_0-2q(E')} |\Phi_n(\vec y,\vec t, \vec s)|\leq \pi^{2d}\max_{\vec x\in[-\pi,\pi]^d}|g(\vec x)|^2 \, h(\vv y),
\eque 
with
\begin{equation}\label{def:h}
h(\vec y):=\sup_{\vec x\in\H_<}|\log\psi(\vec x+\pi_\ge\vec y)|^{-2}\,\prod_{k\in I_\le}\min\left\{1,\frac{1}{|y_k|^2}\right\}.
\end{equation}
Applying the dominated convergence theorem to~\eqref{eq:bntbns},~\eqref{eq:phi_n} and~\eqref{eq:DCT2} and using \eqref{eq:C}, to show the desired result it remains to prove that $h$ is integrable on $\Rd$.

Formally, write
\begin{multline*}
\int_{\R^d}h(\vec y)d\vec y = \int_{\H_\geq}\int_{\H_<}h(\vec y)\,d\lambda_<\otimes\lambda_\geq(\vec y)\\
= \int_{\H_<}\prod_{k\in I_<}\min\ccbb{1,\frac1{|y_k|^2}}\,d\lambda_<(\vec y)\int_{\H_\geq}\sup_{\vec x\in\H_<}|\log\psi(\vec x+\vec y)|^{-2}\prod_{k\in I_=}\min\ccbb{1,\frac1{|y_k|^2}}\,d\lambda_\geq(\vec y).
\end{multline*}
By Fubini's theorem, $h$ is integrable over $\R^d$ if 
\begin{equation}\label{int1}
\int_{\H_<}\prod_{k\in I_<}\min\left\{1,\frac{1}{|y_k|^2}\right\}\, d\lambda_<(\vec y) < \infty
\end{equation}
and
\begin{equation}\label{int2}
\int_{\H_\ge}h(\vec y)  \, d\lambda_\ge(\vec y)=\int_{\H_\ge} \sup_{\vec x\in\H_<}|\log\psi(\vec x+\vec y)|^{-2}\prod_{k\in I_=}\min\left\{1,\frac{1}{|y_k|^2}\right\}\, d\lambda_\ge(\vec y) < \infty.
\end{equation}
The integrability condition \eqref{int1} is obvious. For \eqref{int2}, let us remark that the function
$\vec y\in\H_\ge\mapsto\inf_{\vec x\in\H_<}|\log\psi(\vec x+\vec y)|$ is $(\pi_\ge E)$-homogeneous and since $q(\pi_\ge E)>2$, by Lemma~\ref{integrabilite}, the function  $\vec y\in\H_\ge\mapsto\sup_{\vec x\in\H_<}|\log\psi(\vec x+\vec y)|^{-2}$ is locally integrable on $\H_\ge$ with respect to $\lambda_\ge$.
Together with the fact that $\sup_{\vec x\in\H_<}|\log\psi(\vec x+\vec y)|^{-2}$ is bounded by $1$ for $\pi_=\vec y$ large enough, this shows that
$$
\int_{\H_\ge}\indd{\|\pi_>\vec y\|_{_{\pi_{>}E}}\le 1} h(\vec y)\, d\lambda_\ge(\vec y) < \infty,
$$
with the definition of  $\|\cdot\|_{_{\pi_>E}}$ given in \eqref{def:Enorm}. Moreover, 
\begin{align*}
&\hspace{-20pt}\int_{\H_\ge}\indd{\|\pi_>\vec y\|_{\pi_{>}E}> 1} h(\vec y)\, d\lambda_\ge(\vec y)\\
&\le
\int_{\H_>} \indd{\|\vec y\|_{_{\pi_{>}E}}> 1}\sup_{\vec x\in\H_\le}|\log\psi(\vec x+\vec y)|^{-2}\,d\lambda_>(\vec y)
\,\int_{\H_=}\prod_{k\in I_=}\min\left\{1,\frac{1}{|y_k|^2}\right\}\,d\lambda_=(\vec y).
\end{align*}
The second integral is clearly finite. For the first one,
since $\vec y\in\H_>\mapsto\inf_{\vec x\in\H_\le}|\log\psi(\vec x+\vec y)|$ is $(\pi_> E)$-homogeneous and  $q(\pi_>E)<2$, one has
\begin{align*}
&\hspace{-20pt}\int_{\H_>}\indd{\|\vec y\|_{_{\pi_{>}E}}> 1} \sup_{\vec x\in\H_\le}|\log\psi(\vec x+\vec y)|^{-2}\,d\lambda_>(\vec y)\\
&=\int_1^{+\infty}r^{q(\pi_>E)-3}\int_{S_{\pi_>E}}\sup_{\vec x\in\H_\le}|\log\psi(\vec x+\vec \theta)|^{-2}\, d\sigma_{\pi_>E}(\vec \theta)<\infty,
\end{align*}
where $S_{\pi_>E}$ is the unit sphere of $\H_>$ with respect to $\|\cdot\|_{\pi_>E}$ and $\sigma_{\pi_>E}$ is the Radon measure on $S_{\pi_>E}$ such that $d\lambda_>=r^{q(\pi_>E)-1}drd\sigma_{\pi_>E}$.
This shows that \eqref{int2} holds and thus the function $h$ in \eqref{def:h} is integrable over $\R^d$.
\end{proof}

\subsection{Finite-dimensional convergence}\label{sec:fdd}\

We start by showing that the coefficients $b_{n,\vec j}(\vec t)$ defined in \eqref{b_nj} satisfy the condition \eqref{eq:sup} of Theorem~\ref{thm:clt} in the following lemma.
\begin{lemma}\label{lem:translation}
For all $\vec t \in (0,1]^d$ and all $q=1,\dots,d$, 
\begin{equation*}
\lim_{{n}\to\infty}\frac1{\|b_{n}(\vec t)\|^2}\sum_{{\vec j}\in\Z^d}|{b_{{n},{\vec j}}^2(\vec t) - b_{{n}, {\vec j}+{\vec e}_q}^2(\vec t)}| = 0
\end{equation*}
and \eqref{eq:sup} holds.
\end{lemma}
\begin{proof}
Fix  $\ell\in\{1,\dots,d\}$ and $t\in(0,1]^d$ be fixed. Using Cauchy--Schwarz inequality,
\begin{align*}
\sum_{{\vec j}\in\Z^d}|{b_{{n},{\vec j}}^2(\vec t) - b_{{n}, {\vec j}+{\vec e}_\ell}^2(\vec t)}|
\le \left(\sum_{{\vec j}\in\Z^d}(b_{{n},{\vec j}}(\vec t) - b_{{n}, {\vec j}+{\vec e}_\ell}(\vec t))^2\right)^{\frac12} 2\|b_n(\vec t)\|.
\end{align*}
So, it is enough to show that
$$
\sum_{{\vec j}\in\Z^d}(b_{{n},{\vec j}}(\vec t) - b_{{n}, {\vec j}+{\vec e}_\ell}(\vec t))^2=o(\|b_n(\vec t)\|^2).
$$
But, we have
\begin{align*}
b_{{n},{\vec j}}(\vec t) - b_{{n}, {\vec j}+{\vec e}_\ell}(\vec t)= \sum_{\substack{{\vec k}\in[{\vec 0},n^{E'}{\vec t}-{\vec 1}]\\ \text{with } k_\ell=\lfloor n^{1/\alpha'_\ell}t_\ell\rfloor-1}}q_{{\vec k}-{\vec j}}\,-\,\sum_{\substack{{\vec k}\in[{\vec 0},n^{E'}{\vec t}-\vec 1]\\ \text{with } k_\ell=0}}q_{{\vec k}-{\vec j-\vec e_\ell}}.
\end{align*}
Thus, 
\begin{align*}
\sum_{{\vec j}\in\Z^d}(b_{{n},{\vec j}}(\vec t)  - b_{{n}, {\vec j}+{\vec e}_\ell}(\vec t) )^2
&\le 2\sum_{{\vec j}\in\Z^d}\left(\sum_{\substack{{\vec k}\in[{\vec 0},n^{E'}{\vec t}-{\vec 1}]\\ \text{with } k_\ell=0}}q_{{\vec k}-{\vec j}}\right)^2.
\end{align*}
Let $\varepsilon >0$. Using Lemma~\ref{lem:cov}, we get
\begin{align*}
\limsup_{n\to\infty}\frac1{\|b_n(\vec t)\|^2}\sum_{{\vec j}\in\Z^d}\left(\sum_{\substack{{\vec k}\in[{\vec 0},n^{E'}{\vec t}-{\vec 1}]\\ \text{with } k_\ell=0}}  q_{{\vec k}-{\vec j}} \right)^2
&\le \limsup_{n\to\infty}\frac1{\|b_n(\vec t)\|^2}\sum_{{\vec j}\in\Z^d}\left(\sum_{\substack{{\vec k}\in[{\vec 0},n^{E'}{\vec t}-{\vec 1}]\\ \text{with } k_\ell\le \varepsilon n^{1/\alpha'_\ell}t_\ell-1}}  q_{{\vec k}-{\vec j}} \right)^2\\
&=\limsup_{n\to\infty}\frac{\|b_n(t_1,\ldots,t_{\ell-1},\varepsilon t_\ell,t_{\ell+1},\ldots,t_d)\|^2}{\|b_n(\vec t)\|^2}\\
&=\frac{V(t_1,\ldots,t_{\ell-1},\varepsilon t_\ell,t_{\ell+1},\ldots,t_d)}{V(\vec t)},
\end{align*}
where $V(\vec t):=\wb C(\vec t, \vec t)$ with the covariance function $\wb C(\cdot,\cdot)$ defined in~\eqref{eq:C}.
We conclude the proof of the lemma using that, for any $\vec t \in(0,1]^d$,
$$V(t_1,\ldots,t_{\ell-1},\varepsilon t_\ell,t_{\ell+1},\ldots,t_d)\to0, \mbox{ as }\varepsilon\to 0.$$
The fact that \eqref{eq:sup} holds is a consequence of Lemma~\ref{lem:trans-sup}.
\end{proof}

To prove the finite-dimensional convergence, we use the Cram\`er-Wold device.
Let $m\in\N$, $\vec t_1,\ldots,\vec t_m\in[0,1]^d$, $\lambda_1,\ldots,\lambda_m\in\R$, and consider 
$S_n^{(m)}=\sum_{k=1}^m \lambda_k S^{E'}_{n}(\vec t_k)$. One has
$$
S_n^{(m)}-\E(S_n^{(m)})=\sum_{j\in\Z^d} d_{n,\vec j} X_{\vec j}^*,
$$
where $d_{n,\vec j}:=\sum_{k=1}^m \lambda_k b_{n,\vec j}(\vec t_k)$ and $\Var(S_n^{(m)})=\|d_n\|^2\Var(X_{\vec 0}^*)$. Using Lemma~\ref{lem:cov}, we get
$$
 \|d_n\|^2
=\sum_{k=1}^m\sum_{\ell=1}^m \lambda_k\lambda_\ell \langle b_{n}(\vec t_k),b_{n}(\vec t_\ell) \rangle
\underset{n\rightarrow \infty}{\sim} \frac{n^{2\gamma_0+2q(E')-q(E'')}}{(2\pi)^d} \sum_{k=1}^m\sum_{\ell=1}^m \lambda_k\lambda_\ell \wb C(\vec t_k,\vec t_\ell),
$$
where $\overline C$ is defined in \eqref{eq:C}.

If $\sum_{k=1}^m\sum_{\ell=1}^m \lambda_k\lambda_\ell \wb C(\vec t_k,\vec t_\ell)=0$, then $\frac{1}{n^{\gamma_0+q(E')-q(E'')/2}}(S_n^{(m)}-\E(S_n^{(m)}))$ converges to $0$ in $L^2$.
If $\sum_{k=1}^m\sum_{\ell=1}^m \lambda_k\lambda_\ell \wb C(\vec t_k,\vec t_\ell)>0$, we get that for each $k=1,\ldots,m$, 
$$
\|b_n(\vec t_k)\|^2\underset{n\rightarrow \infty}{\sim} \|d_n\|^2\frac{\wb C(\vec t_k,\vec t_k)}{\sum_{k=1}^m\sum_{\ell=1}^m \lambda_k\lambda_\ell \wb C(\vec t_k,\vec t_\ell)}.
$$
Thus, since the $b_{n,\vec j}(\vec t_k)$ satisfy \eqref{eq:sup},
$$
\sup_{\vec j} |d_{n,\vec j}|\le \sum_{k=1}^m \lambda_k \sup_{\vec j} |b_{n,\vec j}(\vec t_k)|= \sum_{k=1}^m \lambda_k \, o(\|b_{n}(\vec t_k)\|)  = o(\|d_n\|).
$$
This proves that \eqref{eq:sup} also holds for the $d_{n,\vec j}$ and Theorem~\ref{thm:clt} applies to $S_n^{(m)}$. 
We thus proved the finite-dimensional convergence.

\subsection{Tightness}\label{sec:tightness}\

To prove the tightness, by \citet{bickel71convergence}, following \citep{wang14invariance} and \citep{lavancier05processus_TR}, it is enough to show that for some $p>0$ there exist $\gamma>1$ and $C>0$ such that
for all $\vec t=(t_1,\ldots,t_d)\in[0,1]^d$,
$$
\E\left|\frac{S^{E'}_{n}({\vec t})-\E(S^{E'}_{n}({\vec t}))}{n^{\gamma_0+q(E')-\frac{q(E'')}{2}}}\right|^p \le C \prod_{j=1}^d t_j^{\gamma}.
$$

\medskip

Recall from the equation~\eqref{eq:bntbns} that for all $\vec t\in[0,1]^d$, we have
\begin{align*}
\|b_n(\vec t)\|^2
& =  \frac{n^{-q(E'')}}{(2\pi)^d}\int_{n^{E''}[-\pi,\pi]^d} \left|Q(n^{-E''}\vec y)\right|^2\left(\prod_{k=1}^d\left|D_{\lfloor n^{1/\alpha'_k}t_k-1\rfloor}(n^{-\gamma_k/\alpha'_k}y_k)\right|^2\right)\,d\vec y.
\end{align*}
For any $\delta\in(0,1)$, observe that $|\sin^2(x)|=|\sin^{1-\delta}(x)||\sin^{1+\delta}(x)|\le \min\{|x|^{1-\delta},|x|^2\}$ for all $x$, and $|\sin(x)|\ge \frac{2}{\pi}|x|$ for $x\in[-\pi/2,\pi/2]$.
Then, for all $n$ and $y$ such that $|n y|\le \pi$ and all $t\in[0,1]$, one has 
\begin{multline*}
n^{-2}|D_{\lfloor nt-1\rfloor}(n^{-1}y)|^{2}
=\frac{\sin^2\left(\lfloor nt\rfloor\frac{y}{2n}\right)}{n^2\sin^2\left(\frac{y}{2n}\right)}
\\
\le \min\left\{\frac{\pi^2}{2^{1-\delta}}\frac{t^{1-\delta}}{|y|^{1+\delta}},\frac{\pi^2}{4}t^2\right\}
\le \frac{\pi^2}{2^{1-\delta}} t^{1-\delta}\min\left\{\frac{1}{|y|^{1+\delta}},1\right\},
\end{multline*}
and thus,
\begin{equation*}
 n^{-2}|D_{\lfloor nt-1\rfloor}(n^{-\gamma}y)|^{2}\le\left\{
\begin{array}{ll}
\displaystyle\frac{\pi^2}{2^{1-\delta}} t^{1-\delta}\min\left\{\frac{1}{|y|^{1+\delta}},1\right\}&\mbox{ if }\gamma=1\\
\displaystyle \frac{\pi^2}{4}t^2 &\mbox{ if }\gamma>1                                                                                    \end{array}\right..
\end{equation*}
Recalling that $\gamma_k/\alpha'_k>1$ if and only if $k\in I_>$, 
together with \eqref{eq:boundQ}, this shows that there exists a constant $C>0$ such that
\begin{align*}
&\hspace{-20pt}n^{-2\gamma_0-2q(E')+q(E'')} \|b_n(\vec t)\|^2\\
&\le C\left(\prod_{j\in I_>}t_j^2 \right)\left(\prod_{j\in I_\le} t_j^{1-\delta}\right) \int_{\R^d}\sup_{\vec x\in \H_<}|\log\psi(\vec x+\pi_\ge\vec y)|^{-2}\prod_{j\in I_\le}\min\left\{\frac{1}{|y_j|^{1+\delta}},1\right\}d{\vec y}.
\end{align*}
One can show that this last integral is finite by proceeding exactly as we did to show the integrability of the function $h$ in \eqref{def:h}. The important point is that $1+\delta>1$ to guarantee the integrability of $\frac{1}{|y|^{1+\delta}}$ at infinity.
Hence, for a new constant $C'>0$,
\begin{equation*}
n^{-2\gamma_0-2q(E')+q(E'')} \|b_n(\vec t)\|^2
\le C'\left(\prod_{j\in I_>}t_j^2 \right)\left(\prod_{j\in I_\le} t_j^{1-\delta}\right)\le C' \prod_{j=1}^d t_j^{1-\delta} .
\end{equation*}
Let $p>2$. Using Burkh\"older inequality and the preceding inequality,  there exists a constant $c_p>0$ such that
\begin{align*}
 \E\left|\frac{S^{E'}_{n}({\vec t})-\E(S^{E'}_{n}({\vec t}))}{n^{\gamma_0+q(E')-q(E'')/2}}\right|^p 
&\le c_p \E\left( \sum_{\vec j\in\Z^d} \frac{b_{n,\vec j}^2(\vec t)}{n^{2\gamma_0+2q(E')-q(E'')}} X_{\vec j}^{*2}\right)^{\frac p2}\\
&\le c_p \left( \frac{\|b_n(\vec t)\|^2 }{n^{2\gamma_0+2q(E')-q(E'')}} \right)^{\frac p2}\\
&\le c_pC'^{p/2} \prod_{j=1}^d t_j^{(1-\delta)p/2},
\end{align*}
which gives the tightness by choosing $\delta>1-\frac2p$.

\section{Properties of the limit  field}\label{sec:property}

In this section we focus on the zero-mean Gaussian random field $(W(\vec t))_{\vec t \in \R^d}$ arising in the limit in Theorem~\ref{thm:main}. Recall that this random field depends on both $E$ and $E'$.

\subsection{Increments}\

We may consider a harmonizable representation of $W$, defined on the whole space $\R^d$ by 
$$W(\vec t)=\sigma_X\left(\prod_{k\in I_>}t_k\right)\int_{\R^d}\pp{\prod_{k\in I_{\le}}\frac{e^{it_{k} y_k}-1}{iy_k}} |\log\psi(\pi_{\ge}\vec y)|^{-1}\wt{\mathcal M}(d\vec y), \forall \vec t\in \R^d,$$
with $\sigma_X$ given in~\eqref{eq:sigma}, and $\wt{\mathcal M}$ is a centered complex-valued Gaussian measure on $\R^d$ with Lebesgue control measure (see \citep{xiao09sample}). 
The harmonizable representation shows that the random field has stationary rectangular increments. In the sequel we let $(\vec e_1,\ldots,\vec e_d)$  denote the canonical basis of $\R^d$.
Rectangular increments of $W$   are   defined for $\vec s<\vec t$ by
\begin{eqnarray*}
W([\vec s, \vec t])&=&\sum_{\varepsilon\in \{0,1\}^d}(-1)^{d+|\varepsilon|_1}
W(s_1+\varepsilon_1(t_1-s_1),\ldots, s_d+\varepsilon_d(t_d-s_d))\\
&=&\Delta_{t_1-s_1}^{(1)}\Delta_{t_2-s_2}^{(2)}\ldots\Delta_{t_d-s_d}^{(d)}W(\vec s),
\end{eqnarray*}
where $|\varepsilon|_1=\varepsilon_1+\ldots+\varepsilon_d$ and $\Delta_{\delta}^{(j)}$ corresponds to the directional increment of step  $\delta\in\R$ in direction $j$ for $1\le j\le d$, defined by
$$\Delta_\delta^{(j)}W(\vec t)=W(\vec t+\delta \vec e_j)-W(\vec t).$$

A  direct consequence of Theorem \ref{thm:main} are the following   properties of the random field $W$. 
\begin{proposition}\label{prop:increments0} The random field $W$ satisfies the following properties:
\begin{enumerate}[(i)]
\item\label{label:increments}  stationary rectangular increments: for any fixed $\vec s\in\R^d$,
$$(W([\vec s,\vec t]))_{\vec s<\vec t}\stackrel{fdd}{=} (W([\vec 0,\vec t-\vec s]))_{\vec s<\vec t}\equiv(W( \vec t-\vec s))_{\vec s<\vec t};$$
\item $(E',H)$-operator-scaling property: for all $\lambda>0$
$$(W(\lambda^{E'} \vec t))_{\vec t\in \R^d}\overset{fdd}{=}(\lambda^HW(\vec t))_{\vec t\in \R^d},$$
with $H=\gamma_0+q(E')-\frac{q(E'')}{2}$. 
\end{enumerate}
\end{proposition} 
We can say more about the directional increments $\Delta_\delta\topp jW(\vec t)$. First of all, as a special case of Proposition~\ref{prop:increments0}, \iteqref{label:increments},  $W(\vec t)$ viewed as a process indexed by $t_j\in\R$ has stationary increments. Moreover,
simple dependence properties in the directions corresponding to $I_{>}$ and $I_<$, if not empty, are given below. Following ideas from \citep[Definition 2.2]{puplinskaite13aggregation} we state the following proposition. Recall that $|I_>|\leq 1$.

\begin{proposition}\label{prop:increments} The random field $W$ satisfies the following properties:
\begin{enumerate}[(i)]
\item  When $I_{>} = \{j\}$, the random field $W$ has invariant  increments in the direction $\vec e_j$: for all $h, \delta \in\R$, $\vec t\in\R^d$, we have $\Delta_\delta^{(j)}W(\vec t+h\vec e_j)=\Delta_\delta^{(j)}W(\vec t)$.
\item When $I_{<}\neq \emptyset$, the random field $W$ has independent  increments in any direction $\vec e_j$ with $j\in I_{<}$: for all $\delta>0, \vec t\in \R^d$,
$\Delta_\delta^{(j)}W(\vec t)$  is independent from $W(\vec t)$. 
\end{enumerate}
\end{proposition} 
\begin{proof}
Let $\langle\vec e_j\rangle^\perp$ denote the subspace of $\Rd$ orthogonal to $\vec e_j$. Let $\pi_{\langle\vec e_j\rangle^\perp}$ and $\lambda_{\langle\vec e_j\rangle^{\perp}}$ denote the corresponding projection and Lebesgue measure, respectively.
First, let us simply remark that for $I_{>} = \{j\}$, $\delta\in\R$, and  $\vec t\in\R^d$, 
\equh\label{eq:Deltaj}
\Delta_\delta^{(j)}W(\vec t)=\delta W(\pi_{\langle \vec e_j\rangle^{\perp}}(\vec t)+\vec e_j),
\eque
which does not depend on $t_j$. 
The desired statement then follows.
For the second statement,  when $j\in I_{<}$, 
\begin{multline}\label{increments}
\Delta_\delta^{(j)}W(\vec t)=\sigma_X\left(\prod_{k\in I_>}t_k\right)\\
\times \int_{\R^d}
\frac{e^{it_jy_j}\left(e^{i\delta y_j}-1\right)}{iy_j}\left(\prod_{k\in I_{\le};k\neq j}\frac{e^{it_{k} y_k}-1}{iy_k}\right) |\log\psi(\pi_{\ge}\vec y)|^{-1}\wt{\mathcal M}(d\vec y). 
\end{multline}
Therefore
$$\mbox{Cov}(\Delta_\delta^{(j)}W(\vec t),W(\vec t))=C_{\vec e_j}(\vec t)\int_{\R}\frac{\left(e^{i\delta y_j}-1\right)\left(1-e^{i t_j y_j}\right)}{|y_j|^2}dy_j,$$
with $$C_{\vec e_j}(\vec t)=\sigma_X^2\left(\prod_{k\in I_>}t_k\right)^2\int_{\langle \vec e_j \rangle^\perp}\prod_{k\in I_{\le};k\neq j}\left|\frac{e^{it_{k} y_k}-1}{iy_k}\right|^2 |\log\psi(\pi_{\ge}\vec y)|^{-2}d\lambda_{\langle \vec e_j \rangle^\perp}(\vec y).$$
Hence,
$\mbox{Cov}(\Delta_\delta^{(j)}W(\vec t),W(\vec t))=2\pi C_{\vec e_j}(\vec t)\mbox{Cov}(B_{1/2}(t_j+\delta)-B_{1/2}(t_j),B_{1/2}(t_j))$, with $B_{1/2}$ a standard Brownian motion on $\R$. By independent increments of $B_{1/2}$, we obtain that $\mbox{Cov}(\Delta_\delta^{(j)}W(\vec t),W(\vec t))=0$ for $\delta\ge 0$. Since $W$ is a Gaussian field, we conclude that $\Delta_\delta^{(j)}W(\vec t)$  is independent from   $W(\vec t)$. 
\end{proof}

Let us quote that our definitions of invariant and independent increments are not the ones used in  \citep[Definition 2.2]{puplinskaite13aggregation}. However we remark that invariant increments in the direction $e_{\vec j}$ lead to invariant rectangular increments in the sense that,  for all $\delta\in \R$, and $\vec s<\vec t$
$$W([\vec s+\delta \vec e_j,\vec t+\delta \vec e_j]) {=} W([\vec s,\vec t]).$$
This follows from the fact that $$W([\vec s+\delta \vec e_j,\vec t+\delta \vec e_j]) =\Delta_{t_1-s_1}^{(1)}\Delta_{t_2-s_2}^{(2)}\ldots\Delta_{t_d-s_d}^{(d)}W(\vec s+\delta \vec e_j).$$
Indeed, computing first  $\Delta_{t_j-s_j}^{(j)}W(\vec s+\delta e_j)=\Delta_{t_j-s_j}^{(j)}W(\vec s)$, we obtain the desired result.

When the increments are either invariant or independent in at least one direction, we say that $W$ has degenerate increments. Otherwise, we say that $W$ has non-degenerate increments. 

\begin{example}
When $d=2$, choosing $E'=\mbox{diag}(1,\beta)$ for $\beta>0$ as in \citep{puplinskaite13aggregation} we obtain that $|I_{=}|=2$
if and only if $\rho_1=\rho_2$, that is $\beta=\frac{\alpha_2}{\alpha_1}$. It follows that for $\beta\neq \frac{\alpha_2}{\alpha_1}$,
one has $|I_{=}|=1$ and $W$ has either independent or invariant increments in the orthogonal direction.
However, when $\beta=\frac{\alpha_2}{\alpha_1}$ we get
$$W(\vec t)=\sigma_X \int_{\R^2}\pp{\prod_{k=1}^2\frac{e^{it_{k} y_k}-1}{iy_k} }|\log\psi(\vec y)|^{-1}\wt{\mathcal M}(d\vec y), \forall \vec t\in \R^2.$$
In this case, $W$ has non-degenerate increments. Recall that
all possible non-critical cases in $d=2$ have been provided in Theorem~\ref{thm:non-critical} in introduction. 
\end{example}
More generally for $d\geq 2$  we can state the following scaling-transition property.
\begin{corollary} The random field $(X_{\vec j})_{\vec j\in\Z^d}$, defined in Section \ref{randomfield}, exhibits a scaling-transition in the sense that
\begin{enumerate}[(i)]
\item If there exists $c>0$ such that $E'=c E$, then $W$ has non-degenerate  increments;
\item Otherwise, $W$ has degenerate increments. That is, there exists at least one direction in which the increments of the limit random field are either invariant or independent.
\end{enumerate}
\end{corollary}

In the sequel, we need to control the variance of the directional increments.  By Proposition~\ref{prop:increments}, for all $\vec u\in\R^d, \delta\in\R$,
\equh\label{eq:var>}
\Var(\Delta_\delta^{(j)}W(\vec u)) = \delta^2\Var(W(\pi_{\langle \vec e_j\rangle^{\perp}}(\vec u)+\vec e_j)), j\in I_>,
\eque and
\equh\label{eq:var<}
\Var(\Delta_\delta^{(j)}W(\vec u)) = |\delta|\Var(W(\pi_{\langle \vec e_j\rangle^{\perp}}(\vec u)+\vec e_j)), j\in I_<.
\eque
The control for $j\in I_=$ is a little more involved, as summarized in the following lemma.  
\begin{lemma}\label{lem:increments}
There exist some constants $C$ such that for all $\vec u\in[-1,1]^d,\delta\in\R$, $j\in I_=$, the following inequalities hold.
\begin{enumerate}[(a)]
\item \label{item:a}
If $|I_>|=1$ or $I_>=\emptyset$ and $\alpha_j<1/2$, 
\equh\label{eq:var=<}
\Var(\Delta_\delta^{(j)}W(\vec u))\leq C|\delta|^{2\beta_j} \mbox{ with } \beta_j = \alpha_j\pp{1-\frac{q(\pi_>E)}2}+\frac12.
\eque
\item \label{item:b} If $I_> = \emptyset, \alpha_j = 1/2$, then 
\equh\label{eq:var==}
\Var(\Delta_\delta^{(j)}W(\vec u)) \leq C\max(\delta^2,|\delta|^{2H_j}) \mbox{ for all } H_j\in(0,1).
\eque
\item \label{item:c} If $I_> = \emptyset, \alpha_j>1/2$, then 
\equh\label{eq:var=>}
\Var(\Delta_\delta\topp jW(\vec u))\leq C\delta^2.
\eque
\end{enumerate}
\end{lemma}
\begin{proof}
Recall~\eqref{eq:covfBm}.
 For $j\in I_{=}$,  for all $\vec u\in[-1,1]^d$ and $\delta\in\R$,
$$\Var\left(\Delta_\delta^{(j)}W(\vec u)\right)=\left(\sigma_X\prod_{k\in I_>}u_k\right)^2\int_{\R}\left|\frac{e^{i\delta y_j}-1}{iy_j} \right|^2f_j(y_j)dy_j,$$
with $$f_j(y_j)=\int_{\langle \vec e_j \rangle^\perp}\prod_{k\in I_{\le};k\neq j}\left|\frac{e^{iu_{k} y_k}-1}{iy_k}\right|^2 |\log\psi(\pi_{\ge}(\vec y+y_j\vec e_j))|^{-2}d\lambda_{\langle \vec e_j \rangle^\perp}(\vec y).$$
This is a locally integrable function over $\R$ for all values of $\alpha_j\in(0,1)$ due to the fact that $|\log\psi(\pi_\ge\vec y)|$ is a $\pi_\ge E$-homogeneous function, $q(\pi_\geq E)>2$, and Lemma~\ref{integrabilite}.
Furthermore, by $E$-homogeneity and polar coordinate $\vec x = \tau(\vec x)^E\vec\theta(\vec x)$,
\eqnh
|\log\psi(\vec x)|\inv & = & \frac{|\log\psi(\pi_>\vec x)|+|x_j|^{\alpha_j}}{|\log\psi(\vec x)|} (|\log\psi(\pi_>\vec x)|+|x_j|^{\alpha_j})\inv\\ 
& = & \frac{\tau(\vec x)|\log\psi(\pi_>\vec\theta(\vec x))|+\tau(\vec x)|\theta_j(\vec x)|^{\alpha_j}}{\tau(\vec x )|\log\psi(\vec\theta(\vec x))|} (|\log\psi(\pi_>\vec x)|+|x_j|^{\alpha_j})\inv\\
& \leq & c_1(|\log\psi(\pi_>\vec x)|+|x_j|^{\alpha_j})\inv
\eqne
with $c_1 = \max_{\vec\theta\in S_E}(|\log\psi(\pi_>\vec\theta)|+|\theta_j|^{\alpha_j})/|\log\psi(\vec\theta)|$. 
Thus,
\eqnh
f_j(y_j) & \leq & c_1^2\int_{\langle\vec e_j\rangle^\perp}\prod_{k\in I_\leq;k\neq j}\abs{\frac{e^{iu_ky_k}-1}{iy_k}}^2\pp{|\log\psi(\pi_>\vec y)|+|y_j|^{\alpha_j}}^{-2}d\lambda_{\langle\vec e_j\rangle^\perp}(\vec y)\\
& = & c_1^2\pp{\prod_{k\in I_\leq;k\neq j}\int_\R\abs{\frac{e^{iu_ky_k}-1}{iy_k}}^2dy_k}\int_{\H_>}\pp{|\log\psi(\pi_>\vec y)|+|y_j|^{\alpha_j}}^{-2}d\lambda_>(\vec y),
\eqne
and the second integral in the right-hand side above equals
\begin{multline*}
\int_{\H_>}|y_j|^{-2\alpha_j}\pp{\abs{\log\psi((|y_j|^{\alpha_j})^{-E}\pi_>\vec y)}+1}^{-2}d\lambda_>(\vec y)\\
= |y_j|^{-2\alpha_j+\alpha_jq(\pi_>E)}\int_{\H_>}(|\log\psi(\pi_>\vec y)|+1)^{-2}d\lambda_>(\vec y) =:|y_j|^{-2\beta_j+1}c_2
\end{multline*}
with $\beta_j = \alpha_j(1-q(\pi_>E)/2)+1/2$. We have thus obtained
\[
f_j(y_j)\leq c_3|y_j|^{-2\beta_j+1} \mbox{ with } c_3 = c_1^2c_2\prod_{k\in I_\leq;k\neq j}(2\pi u_k).
\]
Recall that $|I_>|\leq 1$.

\noindent \iteqref{item:a} In case that $|I_>|=1$, $q(\pi_>E)>1$ and thus $\beta_j<1$. Therefore by the above calculation and~\eqref{eq:covfBm}, 
\equh\label{eq:Hj}
\Var(\Delta_{\delta}^{(j)}W(\vec u^{(j)})) \leq \sigma_X^2c_3\int\abs{\frac{e^{i\delta y_j}-1}{iy_j}}^2|y_j|^{-2\beta_j+1}dy_j = \sigma_X^2c_3C_{\beta_j}|\delta|^{2\beta_j}.
\eque
In case that $|I_>|=0$, $\beta_j = \alpha_j + 1/2$. If $\alpha_j<1/2$, then the same bound~\eqref{eq:Hj} holds. 

\noindent\iteqref{item:b} If $\alpha_j = 1/2$, then for any $H_j\in(0,1)$,
\begin{eqnarray*}
\int_{\R}\left|\frac{e^{i\delta y_j}-1}{iy_j}\right|^2f_j(y_j)dy_j&\le & \delta^2\int_{|y_j|\le 1}f_j(y_j)dy_j+c_3\int_{|y_j|>1}\left|\frac{e^{i\delta y_j}-1}{iy_j}\right|^2|y_j|^{-2H_j+1}dy_j\\
&\le & c_4\max(\delta^2,|\delta|^{2H_j}),
\end{eqnarray*}
with $$c_4=\max_{\vec u \in [-1,1]^d}\int_{\langle \vec e_j \rangle^\perp}\prod_{k\in I_{\le};k\neq j}\left|\frac{e^{iu_{k} y_k}-1}{iu_k}\right|^2 |\log\psi(\pi_{\ge}\vec y)|^{-2}d\lambda_{\langle \vec e_j \rangle^\perp}(\vec y)+c_3C_{H_j}.$$
Therefore, $$\Var\left(\Delta_{\delta}^{(j)}W(\vec u^{(j)})\right)\le \sigma_X^2c_4\max(\delta^2, |\delta|^{2H_j}).$$
\iteqref{item:c} At last, if $\alpha_j>1/2$, then $\beta_j>1$, the function $f_j$ is integrable on $\R$ and
$$\int_{\R}\left|\frac{e^{i\delta y_j}-1}{iy_j}\right|^2f_j(y_j)dy_j\le \delta^2 \int_{\R}f_j(y_j)dy_j.$$
It then follows that
$$\Var\left(\Delta_{\delta}^{(j)}W(\vec u^{(j)})\right)\le c_5\delta^{2},$$
with $$c_5=\sigma_X^2\sup_{\vec u\in [-1,1]^d}\int_{\R^d}\prod_{k\in I_{\le};k\neq j}\left|\frac{e^{iu_{k} y_k}-1}{iy_k}\right|^2 |\log\psi(\pi_{\ge}\vec y)|^{-2}d\vec y.$$
\end{proof}

\subsection{Fractional Brownian sheets}\
Here we give a complete characterization of when $W$ is a fractional Brownian sheet. 
Recall that a zero-mean Gaussian random field $(X(\vec t))_{\vec t \in \R^d}$ is a standard fractional Brownian sheet with Hurst index $(H_1,\ldots,H_d)\in(0,1]^d$ if
\begin{equation}\label{eq:fBs}
\Cov(X(\vec t),X(\vec s))=\frac{1}{2^d}\prod_{i=1}^d\pp{|t_i|^{2H_i}+|s_i|^{2H_i}-|t_i-s_i|^{2H_i}}.
\end{equation}
Remark that we include the degenerate case that Hurst index equals 1. 

For the limit random field $W$, the covariance function can be factorized according to different directions as
\[
\Cov(W(\vec t),W(\vec s))=\frac{\sigma_X^2}{(2\pi)^{|I_>|}} \prod_{k\in I_<}\Cov(B_{1/2}(t_k),B_{1/2}(s_k))\cdot\prod_{k\in I_>}t_ks_k\cdot\Psi(\vec t,\vec s),
\]
with $\Psi(\vec t,\vec s)$ only depending on $\{t_k,s_k\}_{k\in I_=}$, given by
\[
\Psi(\vec t,\vec s) := \int_{\H_\ge}|\log\psi(\vec y)|^{-2}\prod_{k\in I_=}\frac{(e^{it_ky_k}-1)\wb{(e^{is_ky_k}-1)}}{2\pi |y_k|^2}d\lambda_\ge(\vec y).
\]
Recall $C_H$ in~\eqref{eq:covfBm}.

\begin{proposition}\label{prop:fBs}
Then $W$ is a fractional Brownian sheet, if and only if $|I_=|=1$. In this case, $\Psi(\vec t,\vec s)$ has the following expressions:
in case $I_==\{j\}, I_>=\emptyset$,
\equh\label{eq:psi1}
\Psi(\vec t,\vec s)
 = |\log\psi(\vec e_j)|^{-2}C_{\alpha_j+1/2}\Cov(B_{\alpha_j+1/2}(t_j),B_{\alpha_j+1/2}(s_j));
\eque
in case $I_==\{j\}, I_> = \{k\}$,
\equh\label{eq:psi2}
\Psi(\vec t,\vec s) = \int_{\H_>}|\log\psi(\vec y+\vec e_j)|^{-2}d\lambda_>(\vec y) C_{H_j}\Cov(B_{H_j}(t_j),B_{H_j}(s_j)),
\eque
with $H_j = \alpha_j(1-1/(2\alpha_k))+1/2$. 

\end{proposition}
\begin{proof}
We first prove the `if part'. 
Suppose $I_= = \{j\}$.  In the case  $I_>=\emptyset$, 
\eqnh
\Psi(\vec t,\vec s) 
& = &\int_{\R}|\log\psi(y_j\vec e_j)|^{-2}\frac{(e^{it_jy_j}-1)\wb{(e^{is_jy_j}-1)}}{2\pi |y_j|^2}dy_j\\
& = &\int_\R|\log\psi((|y_j|^{\alpha_j})^E\vec e_j)|^{-2}\frac{(e^{it_jy_j}-1)\wb{(e^{is_jy_j}-1)}}{2\pi |y_j|^2}dy_j\\
& = &\int_\R|\log\psi(\vec e_j)|^{-2}\frac{(e^{it_jy_j}-1)\wb{(e^{is_jy_j}-1)}}{2\pi |y_j|^{2+2\alpha_j}}dy_j.
\eqne
Thus by~\eqref{eq:covfBm}, in case $I_==\{j\}, I_>=\emptyset$,~\eqref{eq:psi1} follows.
In the case $I_>\neq\emptyset$, 
\eqnh
\Psi(\vec t,\vec s) & = & \int_{\R}\int_{\H_>}|\log\psi(\vec y+y_j\vec e_j)|^{-2}\frac{(e^{it_jy_j}-1)\wb{(e^{is_jy_j}-1)}}{2\pi|y_j|^2}d\lambda_>(\vec y)dy_j\\
& = & \int_\R\int_{\H_>}|y_j|^{-2\alpha_j}|\log\psi((|y_j|^{-\alpha_j})^{E}(\vec y+y_j\vec e_j))|^{-2}\frac{(e^{it_jy_j}-1)\wb{(e^{is_jy_j}-1)}}{2\pi|y_j|^2}d\lambda_>(\vec y)dy_j\\
& = & \int_{\H_>}|\log\psi(\vec y+\vec e_j)|^{-2}d\lambda_>(\vec y)\int_\R \frac{(e^{it_jy_j}-1)\wb{(e^{is_jy_j}-1)}}{2\pi|y_j|^{2+2\alpha_j-\alpha_jq(\pi_>E)}}dy_j.
\eqne
That is, in case $I_==\{j\}, I_>\neq\emptyset$, for $H_j = \alpha_j(1-q(\pi_>E)/2)+1/2$,~\eqref{eq:psi2} follows.

Next, we prove the `only if part'. Suppose $W$ is a fractional Brownian sheet with Hurst indices $H_1,\dots,H_d$. From Proposition~\ref{prop:increments0}, $W$ is also $(E',H)$-operator-scaling with $H = \gamma_0+q(E')-q(E'')/2$. Then, it follows that
\[
\frac{H_1}{\alpha_1'}+\cdots+\frac{H_d}{\alpha_d'} = \gamma_0 + q(E')-q(E'')/2,
\]
or equivalently
\equh\label{eq:H}
\sum_{k\in I_\leq}\frac1{\alpha_k'}(H_k-1/2) + \sum_{k\in I_>}\frac1{\alpha_k'}(H_k-1) = \gamma_0\pp{1-\frac12\sum_{k\in I_>}\frac1{\alpha_k}}.
\eque
We consider the variance. By the assumption that $W$ is a fractional Brownian sheet, and the fact that $W$ has stationary directional increments, for all $j\in\{1,\dots,d\}$, for all $\delta\in\R$,
\equh\label{eq:varfBsj}
\Var(\Delta_\delta\topp jW(\vec u)) = |\delta|^{2H_j}\Var(W(\pi_{\langle \vec e_j\rangle^\perp}(\vec u) + \vec e_j)).
\eque 
Recall that $|I_>|\leq 1$. We first consider the case $I_>=\emptyset$. In this case,
\begin{itemize}
\item for $k\in I_<$, comparing~\eqref{eq:varfBsj} and~\eqref{eq:var<} yields $H_k= 1/2$,
\item for $k\in I_=, \alpha_k<1/2$, comparing~\eqref{eq:varfBsj} and~\eqref{eq:var=<} yields $H_k = \alpha_k+1/2$,
\item for $k\in I_=, \alpha_k=1/2$, comparing~\eqref{eq:varfBsj} and~\eqref{eq:var==} yields $H_k = 1$,
\item for $k\in I_=, \alpha_k>1/2$, comparing~\eqref{eq:varfBsj} and~\eqref{eq:var=>} yields $H_k = 1$.
 \end{itemize}
Then~\eqref{eq:H} becomes
 \[
 \sum_{k\in I_=, \alpha_k>1/2}\frac{\gamma_0}{2\alpha_k}+\sum_{k\in I_=, \alpha_k\leq 1/2}\gamma_0 = \gamma_0.
 \]
 Since $\alpha_k<1$, it then follows that $|I_=|=1$. Similarly, in the case $I_>\neq\emptyset$, say $I_> = \{1\}$, it follows from comparing the corresponding inequalities that
 \begin{itemize}
 \item $H_1 = 1$,
 \item for $k\in I_<$, $H_k = 1/2$,
 \item for $k\in I_=$, $H_k = \alpha_k(1-1/(2\alpha_1))+1/2$. 
 \end{itemize}
 Then,~\eqref{eq:H} becomes 
 \[
 \sum_{k\in I_=}\gamma_0\pp{1-\frac1{2\alpha_1}} = \gamma_0\pp{1-\frac1{2\alpha_1}},
 \]
 which implies $|I_=| = 1$. 
\end{proof}
\begin{remark}
When the limit is a fractional Brownian sheets, in directions corresponding to $I_>$, $I_<$ (if not empty) and $I_=$, the Hurst indices equals $1$, $1/2$ and some value in $(1/2,1)$, respectively. Thus, $W$ exhibits long-range dependence in the directions corresponding to $I_\geq$.
\end{remark}
As a concrete example, we prove Theorem~\ref{thm:non-critical}.
\begin{proof}[Proof of Theorem~\ref{thm:non-critical}]
\noindent{Case~\iteqref{d=2:1}:} when $\alpha_2'>\alpha_2, \alpha_2\in(0,1/2)$. In this case, $\gamma_0 = \rho_2=\alpha_2/\alpha_2'$, $E'' = {\rm diag}(1/\alpha_1,1/\alpha_2')$, $I_< = \{1\}, I_= = \{2\}$, $\beta = \alpha_2/\alpha_2'+\frac12(\frac1{\alpha_1}+\frac1{\alpha_2'})$ and $H_1 = 1/2$ are straight-forward. Then, by~\eqref{eq:psi1}, $H_2 = \frac12+\alpha_2$ and  $\sigma^2 = C_{H_2}|\log\psi(0,1)|^{-2}$.\medskip

\noindent{Case~\iteqref{d=2:2}:} when $\alpha_2'>\alpha_2, \alpha_2\in(1/2,1)$. In this case, $\gamma_0 = \rho_1=1$, $E'' = E$, $I_> = \{2\}, I_= = \{1\}$, $\beta = 1+\frac1{2\alpha_1} + \frac1{\alpha_2'} - \frac1{2\alpha_2}$ and $H_2 = 1$ are straight-forward. Then, by~\eqref{eq:psi2}, $H_1 = \frac12+\alpha_1(1-\frac1{2\alpha_2})$ and $\sigma^2 = C_{H_1}\int_\R|\log\psi(1,y)|^{-2}dy$.\medskip

The calculation of cases~\iteqref{d=2:3} and~\iteqref{d=2:4} are similar and thus omitted. One obtains $\sigma^2 = C_{H_1}|\log\psi(1,0)|^{-2}$ for case~\iteqref{d=2:3} and $\sigma^2 = C_{H_2}\int_\R|\log\psi(y,1)|^{-2}dy$ for case~\iteqref{d=2:4}.
\end{proof}
\subsection{Sample-path properties}\

We conclude this section by the following general sample-path properties for the random field $W$
that is a consequence of \citep[Proposition 5.3]{bierme09holder}.

\begin{proposition} \label{reg:lim} There exists a  modification $W^*$ of $W$ on $[0,1]^d$ such that for all $\epsilon>0$, almost surely there exists a finite random variable $Z$   such that for all $\vec s,\vec t\in [0,1]^d$,
$$|W^*(\vec t)-W^*(\vec s)|\le Z \rho(\vec t,\vec s){\log(1+\rho(\vec s,\vec t)^{-1})}^{1/2+\epsilon},$$
with $$\rho(\vec s,\vec t)=\sum_{j\in I> }|t_j-s_j| +\sum_{j\in I_< }|t_j-s_j|^{1/2}+\sum_{j\in I_{=} }|t_j-s_j|^{H_j},$$
where for $j\in I_=$, 
\begin{enumerate}[(a)]
\item $H_j = \alpha_j(1-q(\pi_>E)/2)+1/2$ if either $|I_>| = 1$ or $I_>=\emptyset$ and $\alpha_j<1/2$,
\item $H_j$ can take any value in $(0,1)$ if $I_>=\emptyset$ and $\alpha_j =1/2$, and 
\item $H_j = 1$ if $I_> = \emptyset$ and $\alpha_j>1/2$.
\end{enumerate}
\end{proposition}

\begin{proof} Let us consider $E'''$ the diagonal matrix with entries corresponding to $1$ for 
$j\in I_>$, $2$ for $j\in I_<$ and  $1/H_j$ for $j\in I_=$. Let $\tau_{E'''}$ be the radial part with respect to $E'''$ according to \citep[Equation (9)]{bierme09holder}. Let us quote that since $\vec t\mapsto \rho(\vec0,\vec t)$
is $E'''$ homogeneous and strictly positive on $\R^d\smallsetminus\{\vec0\}$, following ideas of \citet[Theorem 3.2]{clausel11explicit}, the function $\vec t\mapsto \rho(\vec0,\vec t)/\tau_{E'''}(\vec t)$ is continuous and strictly positive
on the compact set $S_{E'''}$. It follows that we may find $C,C'>0$ such that for all $\vec t\in\R^d$,
$$C\tau_{E'''}(\vec t)\le \rho(\vec0,\vec t) \le C' \tau_{E'''}(\vec t).$$
Therefore, by \citep[Proposition 5.3]{bierme09holder} (with $\beta=0$), to show  Proposition \ref{reg:lim}  we prove
for $\vec t, \vec s\in [0,1]^d$ that
\equh\label{eq:boundrho}
\sqrt{\mathbb{E}\left((W(\vec t)-W(\vec s))^2\right)}=\sqrt{\mbox{Var}\left(W(\vec t)-W(\vec s)\right)}\le C\rho(\vec s,\vec t).
\eque
For $\vec t, \vec s\in [0,1]^d$, considering as in \citep{lacaux11invariance}, the sequence
($\vec u^{(j)})_{0\le j\le d}$ defined by $\vec u^{(0)}=\vec s$ and $\vec u^{(j+1)}=\vec u^{(j)}+(t_j-s_j)\vec e_j$ for $0\le j\le d-1$, we get $W(\vec t)-W(\vec s)=\sum_{j=1}^d\Delta_{(t_j-s_j)}^{(j)}W(\vec u^{(j)})$. Hence 
$$\sqrt{\mbox{Var}\left(W(\vec t)-W(\vec s)\right)}\le \sum_{j=1}^d\sqrt{\mbox{Var}\left(\Delta_{(t_j-s_j)}^{(j)}W(\vec u^{(j)})\right)}.$$

Now to obtain~\eqref{eq:boundrho}, it suffices to apply the bounds on the directional  increments established in Lemma~\ref{lem:increments}. Observe that in the case $j\in I_=, I_>=\emptyset$, since  $\delta = t_j-s_j\in[-1,1]$,  the right-hand side of~\eqref{eq:var==} becomes $C|\delta|^{2H_j}$. The details are omitted. The proof is thus completed.
 \end{proof}

Let us quote  that we probably could improve this result. Actually, following \citep{xiao09sample}, it is sufficient to get a similar lower bound on the variance on $[\varepsilon,1]^d$ to establish condition ($C_1$), from which Theorem 4.2 follows, saying that the inequality is true for $\epsilon=0$ and $Z$ has finite moments of any order.

\bibliographystyle{apalike}
\bibliography{references}

\def\cprime{$'$} \def\polhk#1{\setbox0=\hbox{#1}{\ooalign{\hidewidth
  \lower1.5ex\hbox{`}\hidewidth\crcr\unhbox0}}}
  \def\polhk#1{\setbox0=\hbox{#1}{\ooalign{\hidewidth
  \lower1.5ex\hbox{`}\hidewidth\crcr\unhbox0}}}
\begin{thebibliography}{}

\bibitem[Basrak and Segers, 2009]{basrak09regularly}
Basrak, B. and Segers, J. (2009).
\newblock Regularly varying multivariate time series.
\newblock {\em Stochastic Process. Appl.}, 119(4):1055--1080.

\bibitem[Bickel and Wichura, 1971]{bickel71convergence}
Bickel, P.~J. and Wichura, M.~J. (1971).
\newblock Convergence criteria for multiparameter stochastic processes and some
  applications.
\newblock {\em Ann. Math. Statist.}, 42:1656--1670.

\bibitem[Bierm{\'e} and Durieu, 2014]{bierme14invariance}
Bierm{\'e}, H. and Durieu, O. (2014).
\newblock Invariance principles for self-similar set-indexed random fields.
\newblock {\em Trans. Amer. Math. Soc.}, 366(11):5963--5989.

\bibitem[Bierm{\'e} et~al., 2010]{bierme10selfsimilar}
Bierm{\'e}, H., Estrade, A., and Kaj, I. (2010).
\newblock Self-similar random fields and rescaled random balls models.
\newblock {\em J. Theoret. Probab.}, 23(4):1110--1141.

\bibitem[Bierm{\'e} and Lacaux, 2009]{bierme09holder}
Bierm{\'e}, H. and Lacaux, C. (2009).
\newblock H\"older regularity for operator scaling stable random fields.
\newblock {\em Stochastic Process. Appl.}, 119(7):2222--2248.

\bibitem[Bierm{\'e} et~al., 2007]{bierme07operator}
Bierm{\'e}, H., Meerschaert, M.~M., and Scheffler, H.-P. (2007).
\newblock Operator scaling stable random fields.
\newblock {\em Stochastic Process. Appl.}, 117(3):312--332.

\bibitem[Bingham et~al., 1987]{bingham87regular}
Bingham, N.~H., Goldie, C.~M., and Teugels, J.~L. (1987).
\newblock {\em Regular variation}, volume~27 of {\em Encyclopedia of
  Mathematics and its Applications}.
\newblock Cambridge University Press, Cambridge.

\bibitem[Bolthausen, 1982]{bolthausen82central}
Bolthausen, E. (1982).
\newblock On the central limit theorem for stationary mixing random fields.
\newblock {\em Ann. Probab.}, 10(4):1047--1050.

\bibitem[Clausel and Vedel, 2011]{clausel11explicit}
Clausel, M. and Vedel, B. (2011).
\newblock Explicit construction of operator scaling {G}aussian random fields.
\newblock {\em Fractals}, 19(1):101--111.

\bibitem[Davydov, 1970]{davydov70invariance}
Davydov, J.~A. (1970).
\newblock The invariance principle for stationary processes.
\newblock {\em Teor. Verojatnost. i Primenen.}, 15:498--509.

\bibitem[Dedecker, 2001]{dedecker01exponential}
Dedecker, J. (2001).
\newblock Exponential inequalities and functional central limit theorems for a
  random fields.
\newblock {\em ESAIM Probab. Statist.}, 5:77--104 (electronic).

\bibitem[Dedecker et~al., 2011]{dedecker11invariance}
Dedecker, J., Merlev{\`e}de, F., and Peligrad, M. (2011).
\newblock Invariance principles for linear processes with application to
  isotonic regression.
\newblock {\em Bernoulli}, 17(1):88--113.

\bibitem[Deveaux and Fern{\'a}ndez, 2010]{deveaux10partially}
Deveaux, V. and Fern{\'a}ndez, R. (2010).
\newblock Partially ordered models.
\newblock {\em J. Stat. Phys.}, 141(3):476--516.

\bibitem[Enriquez, 2004]{enriquez04simple}
Enriquez, N. (2004).
\newblock A simple construction of the fractional {B}rownian motion.
\newblock {\em Stochastic Process. Appl.}, 109(2):203--223.

\bibitem[Fern{\'a}ndez and Maillard, 2004]{fernandez04chains}
Fern{\'a}ndez, R. and Maillard, G. (2004).
\newblock Chains with complete connections and one-dimensional {G}ibbs
  measures.
\newblock {\em Electron. J. Probab.}, 9:no. 6, 145--176 (electronic).

\bibitem[Fern{\'a}ndez and Maillard, 2005]{fernandez05chains}
Fern{\'a}ndez, R. and Maillard, G. (2005).
\newblock Chains with complete connections: general theory, uniqueness, loss of
  memory and mixing properties.
\newblock {\em J. Stat. Phys.}, 118(3-4):555--588.

\bibitem[Hammond and Sheffield, 2013]{hammond13power}
Hammond, A. and Sheffield, S. (2013).
\newblock Power law {P}\'olya's urn and fractional {B}rownian motion.
\newblock {\em Probab. Theory Related Fields}, 157(3-4):691--719.

\bibitem[Kamont, 1996]{kamont96fractional}
Kamont, A. (1996).
\newblock On the fractional anisotropic {W}iener field.
\newblock {\em Probab. Math. Statist.}, 16(1):85--98.

\bibitem[Kl{\"u}ppelberg and K{\"u}hn, 2004]{kluppelberg04fractional}
Kl{\"u}ppelberg, C. and K{\"u}hn, C. (2004).
\newblock Fractional {B}rownian motion as a weak limit of {P}oisson shot noise
  processes---with applications to finance.
\newblock {\em Stochastic Process. Appl.}, 113(2):333--351.

\bibitem[Kolmogorov, 1940]{kolmogorov40wienersche}
Kolmogorov, A.~N. (1940).
\newblock Wienersche {S}piralen und einige andere interessante {K}urven im
  {H}ilbertschen {R}aum.
\newblock {\em C. R. (Doklady) Acad. Sci. URSS (N.S.)}, 26:115--118.

\bibitem[Lacaux and Marty, 2011]{lacaux11invariance}
Lacaux, C. and Marty, R. (2011).
\newblock From invariance principles to a class of multifractional fields
  related to fractional sheets.
\newblock Preprint, $<$hal-00592188$>$.

\bibitem[Lavancier, 2005]{lavancier05processus_TR}
Lavancier, F. (2005).
\newblock Processus empirique de fonctionnelles de champs gaussiens \`a longue
  m\'emoire.
\newblock {\em PUB.~IRMA, Lille.}, 63(XI):1--26.

\bibitem[Lavancier, 2007]{lavancier07invariance}
Lavancier, F. (2007).
\newblock Invariance principles for non-isotropic long memory random fields.
\newblock {\em Stat. Inference Stoch. Process.}, 10(3):255--282.

\bibitem[Lindskog, 2004]{lindskog04multivariate}
Lindskog, F. (2004).
\newblock Multivariate extremes and regular variation for stochastic processes.
\newblock Ph.D. Thesis, Department of Mathematics, Swiss Federal Institute of
  Technology, Switzerland.

\bibitem[Lodhia et~al., 2014]{lodhia14fractional}
Lodhia, A., Sheffield, S., Sun, X., and Watson, S.~S. (2014).
\newblock Fractional gaussian fields: a survey.
\newblock arXiv preprint arXiv:1407.5598.

\bibitem[Maejima and Mason, 1994]{maejima94operator}
Maejima, M. and Mason, J.~D. (1994).
\newblock Operator-self-similar stable processes.
\newblock {\em Stochastic Process. Appl.}, 54(1):139--163.

\bibitem[Mandelbrot and Van~Ness, 1968]{mandelbrot68fractional}
Mandelbrot, B.~B. and Van~Ness, J.~W. (1968).
\newblock Fractional {B}rownian motions, fractional noises and applications.
\newblock {\em SIAM Rev.}, 10:422--437.

\bibitem[McLeish, 1974]{mcleish74dependent}
McLeish, D.~L. (1974).
\newblock Dependent central limit theorems and invariance principles.
\newblock {\em Ann. Probability}, 2:620--628.

\bibitem[Meerschaert and Scheffler, 2001]{meerschaert01limit}
Meerschaert, M.~M. and Scheffler, H.-P. (2001).
\newblock {\em Limit distributions for sums of independent random vectors}.
\newblock Wiley Series in Probability and Statistics: Probability and
  Statistics. John Wiley \& Sons Inc., New York.
\newblock Heavy tails in theory and practice.

\bibitem[Mikosch and Samorodnitsky, 2007]{mikosch07scaling}
Mikosch, T. and Samorodnitsky, G. (2007).
\newblock Scaling limits for cumulative input processes.
\newblock {\em Math. Oper. Res.}, 32(4):890--918.

\bibitem[Peligrad and Sethuraman, 2008]{peligrad08fractional}
Peligrad, M. and Sethuraman, S. (2008).
\newblock On fractional {B}rownian motion limits in one dimensional
  nearest-neighbor symmetric simple exclusion.
\newblock {\em ALEA Lat. Am. J. Probab. Math. Stat.}, 4:245--255.

\bibitem[Peligrad and Utev, 1997]{peligrad97central}
Peligrad, M. and Utev, S. (1997).
\newblock Central limit theorem for linear processes.
\newblock {\em Ann. Probab.}, 25(1):443--456.

\bibitem[Puplinskait{\.e} and Surgailis, 2013]{puplinskaite13aggregation}
Puplinskait{\.e}, D. and Surgailis, D. (2013).
\newblock Aggregation of autoregressive random fields and anisotropic
  long-range dependence.
\newblock arXiv preprint arXiv:1303.2209.

\bibitem[Puplinskait{\.e} and Surgailis, 2015]{puplinskaite15scaling}
Puplinskait{\.e}, D. and Surgailis, D. (2015).
\newblock Scaling transition for long-range dependent gaussian random fields.
\newblock {\em Stochastic Processes and their Applications}.

\bibitem[Resnick and Greenwood, 1979]{resnick79bivariate}
Resnick, S. and Greenwood, P. (1979).
\newblock A bivariate stable characterization and domains of attraction.
\newblock {\em J. Multivariate Anal.}, 9(2):206--221.

\bibitem[Resnick and Samorodnitsky, 2014]{resnick14tauberian}
Resnick, S. and Samorodnitsky, G. (2014).
\newblock Tauberian theory for multivariate regularly varying distributions
  with application to preferential attachment networks.
\newblock arXiv preprint arXiv:1406.6395.

\bibitem[Resnick, 1987]{resnick87extreme}
Resnick, S.~I. (1987).
\newblock {\em Extreme values, regular variation, and point processes},
  volume~4 of {\em Applied Probability. A Series of the Applied Probability
  Trust}.
\newblock Springer-Verlag, New York.

\bibitem[Resnick, 2007]{resnick07heavy}
Resnick, S.~I. (2007).
\newblock {\em Heavy-tail phenomena}.
\newblock Springer Series in Operations Research and Financial Engineering.
  Springer, New York.
\newblock Probabilistic and statistical modeling.

\bibitem[Samorodnitsky, 2006]{samorodnitsky06long}
Samorodnitsky, G. (2006).
\newblock Long range dependence.
\newblock {\em Found. Trends Stoch. Syst.}, 1(3):163--257.

\bibitem[Samorodnitsky and Taqqu, 1994]{samorodnitsky94stable}
Samorodnitsky, G. and Taqqu, M.~S. (1994).
\newblock {\em Stable non-{G}aussian random processes}.
\newblock Stochastic Modeling. Chapman \& Hall, New York.
\newblock Stochastic models with infinite variance.

\bibitem[Sheffield, 2007]{sheffield07gaussian}
Sheffield, S. (2007).
\newblock Gaussian free fields for mathematicians.
\newblock {\em Probab. Theory Related Fields}, 139(3-4):521--541.

\bibitem[Spitzer, 1976]{spitzer76principles}
Spitzer, F. (1976).
\newblock {\em Principles of random walk}.
\newblock Springer-Verlag, New York-Heidelberg, second edition.
\newblock Graduate Texts in Mathematics, Vol. 34.

\bibitem[Taqqu, 1975]{taqqu75weak}
Taqqu, M.~S. (1975).
\newblock Weak convergence to fractional {B}rownian motion and to the
  {R}osenblatt process.
\newblock {\em Z. Wahrscheinlichkeitstheorie und Verw. Gebiete}, 31:287--302.

\bibitem[Taqqu, 1986]{taqqu86bibliographical}
Taqqu, M.~S. (1986).
\newblock A bibliographical guide to self-similar processes and long-range
  dependence.
\newblock In {\em Dependence in probability and statistics ({O}berwolfach,
  1985)}, volume~11 of {\em Progr. Probab. Statist.}, pages 137--162.
  Birkh\"auser Boston, Boston, MA.

\bibitem[Wang, 2014]{wang14invariance}
Wang, Y. (2014).
\newblock An {I}nvariance {P}rinciple for {F}ractional {B}rownian {S}heets.
\newblock {\em J. Theoret. Probab.}, 27(4):1124--1139.

\bibitem[Xiao, 2009]{xiao09sample}
Xiao, Y. (2009).
\newblock Sample path properties of anisotropic {G}aussian random fields.
\newblock In {\em A minicourse on stochastic partial differential equations},
  volume 1962 of {\em Lecture Notes in Math.}, pages 145--212. Springer,
  Berlin.

\end{thebibliography}
\end{document}